\documentclass{amsart}
\usepackage[utf8]{inputenc}
\usepackage{t1enc}
\usepackage[numbers]{natbib}
\usepackage[polish, french,english]{babel}

\usepackage{graphicx}
\usepackage{sidecap}

\usepackage{subfigure}
\usepackage{relsize}

\usepackage{amsmath}
\usepackage{amsfonts}
\usepackage{amsthm}

\usepackage{float}

\usepackage{lmodern}
\usepackage{xcolor}
\usepackage{titlesec}

\usepackage{cleveref}
\crefname{de}{Definition}{Definitions}
\crefname{prop}{Proposition}{Propositions}
\crefname{lem}{Lemma}{Lemmas}
\crefname{rem}{Remark}{Remarks}
\crefname{ex}{Example}{Examples}
\crefname{ob}{Observation}{Observations}
\crefname{tw}{Theorem}{Theorems}
\crefname{cor}{Corollary}{Corollaries}
\crefname{con}{Conjecture}{Conjectures}
\crefname{figure}{Figure}{Figures}
\crefname{section}{Section}{Sections}
\crefname{equation}{}{Equations}

\usepackage{enumerate}
\usepackage{tikz}
\usepackage{thmtools}
\usepackage{thm-restate}

\title[Algebras with two multiplications and their cumulants]
{Algebras with two multiplications \\ and their cumulants}
\author{Adam Burchardt}
\thanks{Research supported by Narodowe Centrum Nauki, grant number
2014/15/B/ST1/00064}

\titleformat{\section}[block]{\Large\bfseries\filcenter}{\thesection .}{0.5em}{}
\titleformat{\subsection}[runin]{\bfseries}{\thesubsection .}{0.5em}{}
\DeclareMathOperator{\Ch}{Ch}
\DeclareMathOperator{\p}{p}
\DeclareMathOperator{\id}{id}

\usepackage{mathtools}
\DeclarePairedDelimiter{\Kumulant}{\left(\right.}{\left.\right)}
\renewcommand{\c}{\kappa_\nu\Kumulant}

\let\oldk\k
\renewcommand{\k}{\kappa\Kumulant}

\newcommand{\Ph}[1]{{\mathcal{T}} \left( {#1} \right) }
\newcommand{\Pht}[1]{{\mathcal{F}} \left( {#1} \right) }

\newcommand{\Phh}[1]{\hat{\mathcal{F}} \left( {#1} \right) }
\newcommand{\Pto}[1]{\overline{\mathcal{T}} \left( {#1} \right) }
\newcommand{\Pho}[1]{\overline{\mathcal{F}} \left( {#1} \right) }

\newcommand{\Pp}[1]{\mathcal{P} \left( {#1} \right) }

\newcommand{\Po}[1]{\overline{\mathcal{P}} \left( {#1} \right) }

\newcommand{\Poo}[1]{\hat{\mathcal{P}} \left( {#1} \right) }
\newcommand{\Pt}[1]{\mathcal{N} \left( {#1} \right) }

\begin{document}
\maketitle

\newtheorem{tw}{Theorem}[section]
\newtheorem{prop}[tw]{Proposition}
\newtheorem{cor}[tw]{Corollary}
\newtheorem{lem}[tw]{Lemma}
\newtheorem{con}[tw]{Conjecture}
\newtheorem{ob}[tw]{Observation}

\theoremstyle{remark}
\newtheorem{ex}[tw]{Example}
\newtheorem{rem}[tw]{Remark}
\newtheorem{de}[tw]{Definition}

\begin{abstract}
Cumulants are a notion that comes from the classical probability theory, they are an alternative to a notion of moments. 
We adapt the probabilistic concept of cumulants to the setup of a linear space equipped with two multiplication structures.
We present an algebraic formula which involves those two multiplications as a sum of products of cumulants.
In our approach, beside cumulants, we make use of standard combinatorial tools as 
forests and their colourings. We also show that the resulting statement can be understood as an analogue of Leonov--Shiraev's formula. 
This purely combinatorial presentation leads to some conclusions about structure constant of Jack characters.
\end{abstract}

\section{Introduction}

\subsection{Cumulants in probability theory.}
\label{cumulants}

One of classical problems in probability theory is to describe
the joint distribution of a family $\left( X_i \right)$ of random variables in the most convenient
way. Common solution of this problem is to use the family of moments, \emph{i.e.}~the expected values of products of the form
\begin{equation*}
\mathbb{E} \left(  X_{i_1} \cdots X_{i_l} \right) .
\end{equation*}
It has been observed that in many problems it is more convenient to make use of the \emph{cumulants} \cite{Hal81, Fis28},
defined as the coefficients of the expansion of the logarithm of the multidimensional
Laplace transform around zero:
\begin{equation}
\label{cum}
\begin{array}{lrl}

\kappa \left( X_1,\ldots, X_n \right)  &:=& 
\left[ t_1 \cdots t_n \right]  \log \mathbb{E} e^{ t_1 X_1 +\cdots +t_n X_n}  \\
&=&
\dfrac{\partial^n}{\partial t_1 \cdots \partial t_n} \log \mathbb{E} e^{ t_1 X_1 +\cdots +t_n X_n} 
{\Bigg|}_{t_1 = \ldots = t_n =0} ,
\end{array}
\end{equation}

\noindent
where the terms on the right-hand side should be understood as a formal power series in the variables $t_1, \ldots , t_n$. Cumulant is a linear map with respect to each of its arguments.

There are some good reasons for claiming advantage of cumulants over the moments. 
One of them is that the convolution of measures corresponds to the product of the
Laplace transforms or, in other words, to the sum of the logarithms of the
Laplace transforms. It follows that the cumulants behave in a very simple
way with respect to the convolution, namely cumulants linearize the convolution.

Cumulants allow also a combinatorial description. One can show that the
expression \eqref{cum} is equivalent to the following system of equations, called \emph{the moment-cumulant formula}:
\begin{equation}
\label{cumulant}
\mathbb{E} \left( X_1 \cdots X_n \right) =
\sum_\nu \prod_{b \in \nu} \kappa\left( X_i : i \in b \right) 
\end{equation}
\noindent
which should hold for any choice of the random variables $X_1, \ldots , X_n$
whose moments are all finite. The above sum runs over the set partitions
$\nu$ of the set $\left[ n\right] = \lbrace 1,\ldots , n \rbrace$ and the product runs over the blocks of the partition $\nu$. 

\begin{ex}
For three random variables the corresponding moment expands as follows:
\begin{equation*}
\begin{array}{rl}
\mathbb{E}\left( X_{1}X_{2}X_{3}\right)
       =  &\k{X_{1}} \cdot  \k{X_{2}} \cdot  \k{X_{3}}   +        
           \k{X_{1}, X_2} \cdot  \k{X_{3}}                    \\
          &+\k{X_{2}, X_{3}}  \cdot  \k{X_{1}}            +       
           \k{X_{1}, X_{3}} \cdot  \k{X_{2}}                   \\
          & +\k{X_{1}, X_{2}, X_{3}} .                   
\end{array}
\end{equation*}
\end{ex}

\noindent
Observe that the moment-cumulant formula 
defines the cumulant $\k{X_1,\ldots ,X_n}$ 
inductively according to the number of arguments $n$.

\subsection{Conditional cumulants.}
\label{u}

Let $\mathcal{A}$ and $\mathcal{B}$ be commutative unital algebras
and let $\mathbb{E}: \mathcal{A}\longrightarrow \mathcal{B}$ be a unital linear map. We say that $\mathbb{E}$ is a
conditional expected value. 
For any tuple $x_1,\ldots  , x_n \in \mathcal{A}$ we define their \emph{conditional cumulant} as

\begin{equation}
\label{cum2}
\begin{array}{lcl}

\kappa \left( x_1,\ldots, x_n \right)  &=& 
\left[ t_1 \cdots t_n \right]  \log \mathbb{E} e^{ t_1 x_1 +\cdots +t_n x_n}  \\
&=&
\dfrac{\partial^n}{\partial t_1 \cdots \partial t_n} \log \mathbb{E} e^{ t_1 x_1 +\cdots +t_n x_n} 
{\Bigg|}_{t_1 = \ldots = t_n =0} \in \mathcal{B}
\end{array}
\end{equation}

\noindent
where the terms on the right-hand side should be understood as in Eq.~\eqref{cum}.
In this general approach, cumulants give a way of measuring the discrepancy between the algebraic structures of $\mathcal{A}$ and $\mathcal{B}$.

\subsection{Framework.}
In the this paper we are interested in a following particular case. We assume that $\mathcal{A}$ is a linear space equipped with two \emph{commutative multiplication} structures, which correspond to two products: $\cdot$ and $\ast$. Together with each multiplication $\mathcal{A}$ form the commutative algebra. We call such structure an \emph{algebra with two multiplications}. We also assume that the mapping $\mathbb{E}$ is the identity map on $\mathcal{A}$:
\begin{equation*}
\mathbb{E} : \left( \mathcal{A}, \cdot \right) \stackrel{\id}{\longrightarrow} \left( \mathcal{A}, \ast \right) .
\end{equation*}
\noindent
In this case the cumulants measure the discrepancy between these two multiplication structures on $\mathcal{A}$. This situation arises naturally in many branches of algebraic combinatorics, for example in the case of Macdonald cumulants \cite{Dolega2017,Dolega2016b} and cumulants of Jack characters \cite{DolegaFeray2016, Sniady2016}. 

Since the mapping $\mathbb{E}$ is the identity, we can define cumulants of cumulants and further compositions of them. The terminology of cumulants of cumulants was introduced in \cite{MR725217} and further developed in \cite{Lehner2013} (called there \emph{nested cumulants}) in a slightly different situation of an inclusion of algebras 
$\mathcal{C}\subseteq\mathcal{B}\subseteq \mathcal{A}$ and conditional expectations 
$\mathcal{A}\stackrel{\mathbb{E}_1}{\longrightarrow}\mathcal{B}\stackrel{\mathbb{E}_2}{\longrightarrow} \mathcal{C}$.

As we already mentioned in \cref{cumulants}, cumulants allow also a combinatorial description via the moment-cumulant formula. 
When $\mathbb{E}$ is the identity map \cref{cum2} is equivalent to the following system of equations:
\begin{equation}
\label{cumulant2}
 a_1 \ast\cdots\ast a_n  =
\sum_\nu \prod_{b \in \nu} \kappa\left( a_i : i\in b \right)  ,
\end{equation}

\noindent
for any $a_i \in \mathcal{A}$ (the product on the right-hand side is the $\cdot$-product). The above sum runs over the set partitions
$\nu$ of the set $\left[ n\right] $ and the product runs over the blocks of the partition $\nu$. 

Let $A$ be a multiset consisting of elements of the algebra $\mathcal{A}$. To simplify notation, for any partition $\nu$ of a multiset $A $ we introduce the corresponding cumulant $\kappa_\nu$ as the product:
\begin{equation*}
\kappa_\nu = \prod_{b \in \nu} \kappa \left( a: a\in b\right) .
\end{equation*}
\noindent
We denote by $\Pp{A}$ the set of all partitions of $A$. With this notation, the moment-cumulant formula has the following form:
\begin{equation}
\label{cumulant3}
\mathop{\mathlarger{\mathlarger{\mathlarger{\mathlarger{\mathlarger{\ast}}}}}}\limits_{a\in A} a =
\sum_{\nu \in \Pp{A}}  \kappa_\nu .
\end{equation}

\begin{ex}
Given three elements $a_{1} ,a_{2} ,a_{3} \in \mathcal{A}$, we have:
\begin{equation*}
\begin{array}{rcccc}
 a_{1}\ast a_{2}\ast a_{3}
       =  &\k{a_{1}} \cdot  \k{a_{2}} \cdot  \k{a_{3}}   &+        & 
           \k{a_{1}, a_2} \cdot  \k{a_{3}}             &+        \\
          &\k{a_{2}, a_{3}}  \cdot  \k{a_{1}}            &+        &
           \k{a_{1}, a_{3}} \cdot  \k{a_{2}}            &+        \\
          & \k{a_{1}, a_{2}, a_{3}} .                    &         &                                                                 &
\end{array}
\end{equation*}
\end{ex}

\subsection{The main result.}
\label{main subsection}

The purpose of this paper is to present an algebraic formula which involves two multiplications on linear space $\mathcal{A}$:
\begin{equation*}
\left( a^1_1 \ast \cdots \ast a_{k_1}^1 \right)  \cdots \left(  a^n_1 \ast \cdots \ast a_{k_n}^n \right) ,
\end{equation*}
\noindent
as a sum of products of only one type of multiplication.

We use the following notation. 
We denote by $A_1 , \ldots ,A_n$ multisets consisting of elements of $\mathcal{A}$. We denote by $ A = A_1 \cup \cdots \cup A_n$ the multiset, corresponding to the sum of all multisets $A_i$.
We use also the following notation for elements of $A_i$: 
\begin{equation*}
A_i =\left\lbrace   a_1^i ,\ldots ,a_{k_i}^i \right\rbrace  ,
\end{equation*}
\noindent
hence the multiset $A$ consists of the following elements: 
\begin{equation*}
A =\left\lbrace   a_1^1 ,\ldots ,a_{k_1}^1 , \ldots ,  a_1^n ,\ldots ,a_{k_n}^n \right\rbrace   .
\end{equation*}
 
Due to a combinatorial nature of this result we introduce now the definitions of the mixing reduced forests and theirs cumulants. We begin with the following definition.

\begin{de}
\label{root}
Consider a forest $F$ whose leaves are labelled by elements of an algebra $\mathcal{A}$. 
We denote by $A$ the multiset consisting of labels of all leaves.
If each node (vertex which is not a leaf) of $F$, has at least two descendants, we call $F$ a \textit{reduced forest} with leaves in $A$. We denote the set of such forests by $\Pht{A}$ (see \cref{fig0}).
\end{de}

For a reduced forest $F \in \Pht{A}$ we associate a cumulant $\kappa_F$ in the following way:

\begin{de}
\label{xxx}
Consider a reduced forest $F \in \Pht{A}$. Denote by $a_v $ the label of a leaf $v$.
For any vertex $v \in F$ we define inductively the quantities $\kappa_v$ as follows: 
\begin{equation*}
\kappa_v := 
\left\{ \begin{array}{lllll}
a_v & &&                            &$if $v$ is a leaf$,\\
\k{ \kappa_{v_1}, \ldots, \kappa_{v_n}} & &&   &$otherwise,$
\end{array} \right.
\end{equation*}

\noindent
where $v_1 , \ldots, v_n$ are the descendants of $v$. For the whole forest $F$, we define the cumulant $\kappa_F$ to be the product:
\begin{equation*}
\begin{array}{l}
\kappa_F^{} := \mathop{\ast}\limits_{i}^{} \kappa_{V_i}^{} ,
\end{array}
\end{equation*} 
\noindent
where  $V_i$ are the roots of all trees in $F$ (see \cref{fig0}).
\end{de}

Finally, we introduce a class of the \textit{mixing forests} and the associated quantity $w_F$.

\begin{restatable}{de}{definitionw}
\label{kT}
Let us consider a multiset $A = A_1 \cup \cdots \cup A_n$ and a reduced forest $F \in \Pht{A}$. We say that $F$ is \textit{mixing for a division} $A_1 ,\ldots,A_n$ (or shortly \textit{mixing}) if for each vertex $v$ whose descendants are all leaves, those descendants are elements of at least two distinct multisets $A_i$ and $A_j$. 
Denote by $\overline{\mathcal{F}}(A)$ the set of all \emph{reduced mixing forests}. 

For a reduced mixing forest $F$ we define the quantity $w_F$ to be the number of vertices in $F$ minus the number of leaves (see \cref{fig0}).
\end{restatable}

\begin{figure}
\centering
\begin{tikzpicture}[scale=0.8]

\begin{scope}[shift={(12,-6)}]
\draw[gray, thick] (0,0) -- (0.5,-1);
\draw[gray, thick] (1,0) -- (0.5,-1);

\draw[gray, thick] (0.5,-1) -- (1.25,-2);
\draw[gray, thick] (2.5,0) -- (1.25,-2);

\filldraw[color=black, fill=gray] (0,0) circle (2pt) node[anchor=south] {$a_1^1$};
\filldraw[color=black, fill=gray] (1,0) circle (2pt) node[anchor=south] {$a_2^1$};
\filldraw[color=black, fill=white] (2.5,0) circle (2pt) node[anchor=south] {$a_1^2$};

\filldraw[black] (0.5,-1) circle (2pt);

\filldraw[black] (1.25,-2) circle (2pt);
\draw[->,line width=1pt, blue] (1.25,-2.5) to (1.25,-2.1);

\filldraw[black] (-0.5,-3)  node[anchor=west] {$\kappa \Big( \k{a_1^1,a_2^1} ,a_1^2 \Big)$};
\filldraw[black] (0.5,-3.75)  node[anchor=west] {};
\end{scope}

\begin{scope}[shift={(4,-6)}]
\draw[gray, thick] (6.5,0) -- (5.75,-1);
\draw[gray, thick] (5,0) -- (5.75,-1);

\draw[gray, thick] (5.75,-1) -- (5.25,-2);
\draw[gray, thick] (4,0) -- (5.25,-2);

\filldraw[color=black, fill=gray] (4,0) circle (2pt) node[anchor=south] {$a_1^1$};
\filldraw[color=black, fill=gray] (5,0) circle (2pt) node[anchor=south] {$a_2^1$};
\filldraw[color=black, fill=white] (6.5,0) circle (2pt) node[anchor=south] {$a_1^2$};

\filldraw[black] (5.75,-1) circle (2pt);

\filldraw[black] (5.25,-2) circle (2pt);
\draw[->,line width=1pt, blue] (5.25,-2.5) to (5.25,-2.1);

\filldraw[black] (3.5,-3)  node[anchor=west] {$\kappa \Big( a_1^1, \k{a_2^1 ,a_1^2} \Big)$};
\filldraw[black] (4.5,-3.75)  node[anchor=west] {$w_F =2$};
\end{scope}

\begin{scope}[shift={(-4,-6)}]
\draw[gray, thick] (8,0) -- (9.5,-1);
\draw[gray, thick] (10.5,0) -- (9.5,-1);

\draw[gray, thick] (9.5,-1) -- (9.25,-2);
\draw[gray, thick] (9,0) -- (9.25,-2);

\filldraw[color=black, fill=gray] (8,0) circle (2pt) node[anchor=south] {$a_1^1$};
\filldraw[color=black, fill=gray] (9,0) circle (2pt) node[anchor=south] {$a_2^1$};
\filldraw[color=black, fill=white] (10.5,0) circle (2pt) node[anchor=south] {$a_1^2$};

\filldraw[black] (9.5,-1) circle (2pt);

\filldraw[black] (9.25,-2) circle (2pt);
\draw[->,line width=1pt, blue] (9.25,-2.5) to (9.25,-2.1);

\filldraw[black] (7.5,-3)  node[anchor=west] {$\kappa \Big( a_2^1, \k{a_1^1 ,a_1^2} \Big)$};
\filldraw[black] (8.5,-3.75)  node[anchor=west] {$w_F =2$};
\end{scope}

\begin{scope}[shift={(-12,-6)}]
\draw[gray, thick] (12,0) -- (13.25,-1);
\draw[gray, thick] (13,0) -- (13.25,-1);
\draw[gray, thick] (14.5,0) -- (13.25,-1);

\filldraw[color=black, fill=gray] (12,0) circle (2pt) node[anchor=south] {$a_1^1$};
\filldraw[color=black, fill=gray] (13,0) circle (2pt) node[anchor=south] {$a_2^1$};
\filldraw[color=black, fill=white] (14.5,0) circle (2pt) node[anchor=south] {$a_1^2$};

\filldraw[black] (13.25,-1) circle (2pt);
\draw[->,line width=1pt, blue] (13.25,-1.5) to (13.25,-1.1);

\filldraw[black] (12,-3)  node[anchor=west] {$ \k{a_1^1,a_2^1 ,a_1^2} $};
\filldraw[black] (12.5,-3.75)  node[anchor=west] {$w_F =1$};
\end{scope}

\begin{scope}[shift={(12,5)}]
\draw[gray, thick] (0,-6) -- (0.5,-7);
\draw[gray, thick] (1,-6) -- (0.5,-7);

\filldraw[color=black, fill=gray] (0,-6) circle (2pt) node[anchor=south] {$a_1^1$};
\filldraw[color=black, fill=gray] (1,-6) circle (2pt) node[anchor=south] {$a_2^1$};
\filldraw[color=black, fill=white] (2.5,-6) circle (2pt) node[anchor=south] {$a_1^2$};
\draw[->,line width=1pt, blue] (2.5,-6.5) to (2.5,-6.1);

\filldraw[black] (0.5,-7) circle (2pt);
\draw[->,line width=1pt, blue] (0.5,-7.5) to (0.5,-7.1);

\filldraw[black] (0,-8)  node[anchor=west] {$ \k{a_1^1,a_2^1} \ast a_1^2 $};
\filldraw[black] (0.5,-8.75)  node[anchor=west] {};
\end{scope}

\begin{scope}[shift={(4,5)}]
\draw[gray, thick] (6.5,-6) -- (5.75,-7);
\draw[gray, thick] (5,-6) -- (5.75,-7);

\filldraw[color=black, fill=gray] (4,-6) circle (2pt) node[anchor=south] {$a_1^1$};
\draw[->,line width=1pt, blue] (4,-6.5) to (4,-6.1);
\filldraw[color=black, fill=gray] (5,-6) circle (2pt) node[anchor=south] {$a_2^1$};
\filldraw[color=black, fill=white] (6.5,-6) circle (2pt) node[anchor=south] {$a_1^2$};

\filldraw[black] (5.75,-7) circle (2pt);
\draw[->,line width=1pt, blue] (5.75,-7.5) to (5.75,-7.1);

\filldraw[black] (4,-8)  node[anchor=west] {$ a_1^1 \ast \k{a_2^1 , a_1^2} $};
\filldraw[black] (4.5,-8.75)  node[anchor=west] {$w_F =1$};
\end{scope}

\begin{scope}[shift={(-4,5)}]
\draw[gray, thick] (8,-6) -- (9.5,-7);
\draw[gray, thick] (10.5,-6) -- (9.5,-7);

\filldraw[color=black, fill=gray] (8,-6) circle (2pt) node[anchor=south] {$a_1^1$};
\filldraw[color=black, fill=gray] (9,-6) circle (2pt) node[anchor=south] {$a_2^1$};
\draw[->,line width=1pt, blue] (9,-6.5) to (9,-6.1);
\filldraw[color=black, fill=white] (10.5,-6) circle (2pt) node[anchor=south] {$a_1^2$};

\filldraw[black] (9.5,-7) circle (2pt);
\draw[->,line width=1pt, blue] (9.5,-7.5) to (9.5,-7.1);

\filldraw[black] (8,-8)  node[anchor=west] {$a_2^1 \ast \k{a_1^1 ,a_1^2} $};
\filldraw[black] (8.5,-8.75)  node[anchor=west] {$w_F =1$};
\end{scope}

\begin{scope}[shift={(-12,5)}]
\filldraw[color=black, fill=gray] (12,-6) circle (2pt) node[anchor=south] {$a_1^1$};
\draw[->,line width=1pt, blue] (12,-6.5) to (12,-6.1);
\filldraw[color=black, fill=gray] (13,-6) circle (2pt) node[anchor=south] {$a_2^1$};
\draw[->,line width=1pt, blue] (13,-6.5) to (13,-6.1);
\filldraw[color=black, fill=white] (14.5,-6) circle (2pt) node[anchor=south] {$a_1^2$};
\draw[->,line width=1pt, blue] (14.5,-6.5) to (14.5,-6.1);

\filldraw[black] (12,-8)  node[anchor=west] {$a_1^1 \ast a_2^1 \ast a_1^2 $};
\filldraw[black] (12.5,-8.75)  node[anchor=west] {$w_F =0$};
\end{scope}

\draw[-,line width=1pt, black] (0,-4.5) to (14.5,-4.5);
\filldraw[black] (10.4,-4.2)  node[anchor=west] {Not trees};
\filldraw[black] (10.7,-4.8)  node[anchor=west] {Trees};
\draw[-,line width=1pt, black] (11.5,0) to (11.5,-3.9);
\draw[-,line width=1pt, black] (11.5,-5.1) to (11.5,-10);

\end{tikzpicture}
\caption{All reduced forests on $A =A_1 \cup A_2 = \{ a_1^1 , a_2^1  , a_1^2   \} $. 
Six of them (on the right-hand side) are mixing; we present theirs $w_F$ numbers. The remaining two elements (shown on the left-hand side) are not mixing. 
We also present the corresponding cumulants~$\kappa_F$. 
Observe that, among all reduced forests $F\in \Pht{A}$, exactly half, presented on top, consists of a single tree 
(see \cref{rem}). } \label{fig0}
\end{figure}
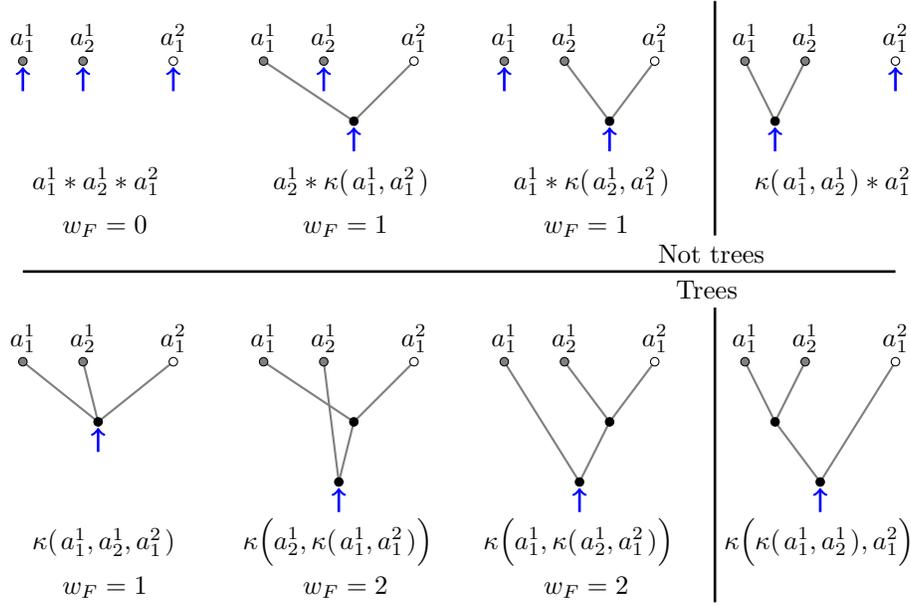

We are ready to formulate the main result of this paper.

\begin{tw}[The main result]
\label{pierwszetwierdzenie}
Let $A_1, \ldots , A_n $ be multisets consisting of elements of $\mathcal{A}$. Let $A$ be the sum of those multisets. Then:
\begin{equation*}
\left( a^1_1 \ast \cdots \ast a_{k_1}^1 \right)  \cdots \left(  a^n_1 \ast \cdots \ast a_{k_n}^n \right) 
= \sum_{F \in \Pho{A}} \left( -1\right)^{w_F^{}} \kappa_F^{} .
\end{equation*}
\end{tw}

\begin{ex}
\label{commuatative}
\cref{fig0} presents all reduced forests $F$ on the multiset $A =\{ a_1^1 , a_2^1  , a_1^2    \}$. 
Six of them are mixing. 
By the statement of the theorem, we have
\begin{equation*}
\begin{array}{lcll}
\left( a_1^1 \ast a_2^1\right) \cdot a_1^2&=& a_1^1 \ast a_2^1\ast a_1^2 - \k{a_1^1 , a_1^2}\ast\k{a_2^1}\\
&& - \k{a_2^1 , a_1^2} \ast \k{a_1^1} - \k{a_1^1 , a_2^1 , a_1^2} \\
&& +\kappa \left( \k{a_1^1 , a_2^2},\k{a_2^1}\right) +\kappa \left( \k{a_2^1 , a_2^2},\k{a_1^1}\right) . 
\end{array}
\end{equation*}
\end{ex}

\subsection{Leonov--Shiryaev's formula.}

In 1959 Leonov and Shiryaev \cite[Equation IV.d]{l-s} presented a formula for a cumulant of products of random variables:
\begin{equation*}
\kappa\left( X_{1,1} \cdots X_{k_1,1} , \ldots ,X_{1,n} \cdots X_{k_n,n}\right)  
\end{equation*}
\noindent
in terms of simple cumulants. This formula was first proved by Leonov and Shiryaev \cite{l-s}, a more direct proof was given by Speed \cite{MR725217}. The technique of Leonov and Shiryaev was used in many situations \cite{MR995615, Lehner2002} 
and was further developed in other papers: Krawczyk and Speicher \cite{MR1757277, Mingo2007} found the free analogue of the formula; the formula was further generalized to the partial cumulants \cite[Proposition 10.11]{MR2266879}.

We briefly present the original formula stated by Leonov and Shiryaev in the framework of an algebra with two multiplications. 
We use the same notation for multisets $A_1 , \ldots ,A_n$ and its sum $ A = A_1 \cup \cdots \cup A_n$ as in \cref{main subsection}. 

We introduce a notion of a strongly-mixing partitions (called also \emph{indecomposable} partitions).

\begin{de}
\label{SMP}
Consider a multiset $A =A_1 \cup \cdots \cup A_n$ and any partition $\nu$ of $A$. 
A partition $\lambda =\{\lambda_1 ,\lambda_2 \}$ is called a \emph{row partition} if for each multiset $A_i$ we have: either $A_i\subseteq \lambda_1$ or $A_i\subseteq \lambda_2$.

A partition $\nu = \{\nu_1 ,\ldots, \nu_q \}$ is called a \textit{strongly-mixing partition for the division} $A =A_1 \cup \cdots \cup A_n$ (or shortly \textit{strongly-mixing partition}), if there is no row partition $\lambda$ such that for any $i$ either $\nu_i \in \lambda_1$, or $\nu_i \in \lambda_2$ (see \cref{fig7}).

We denote by $\Poo{A}$ the set of all strongly-mixing partitions of a set $A$.
\end{de}

\begin{figure}
\centering
\begin{tikzpicture}[scale=1]

\filldraw[gray] (1.8,-4.5) circle (0pt)  node[anchor=west,black] 
{$\nu = \Big\{ \{a_1^1 ,a_1^2 \} ,\{a_1^3 ,a_2^2 ,a_2^3 \}, \{ a_1^4,a_1^5 \} ,\{a_2^5 \}  \Big\}$};
\filldraw[gray] (2.15,-5.5) circle (0pt) node[anchor=west,blue] 
{$\lambda =\Big\{ \{a_1^1 ,a_1^2   ,a_2^2 ,a_1^3,a_2^3 \},\{ a_1^4,a_1^5 ,a_2^5 \} \Big\} $};

\begin{scope}[shift={(0,5.5)}]
\draw[thick,rounded corners=21pt,blue,dashed] (8,-5.5) -- (9.5,-5.5) -- (10.5,-7.5) -- (9.5,-9) -- (7,-9) -- (6,-7.5) -- (7,-5.5) -- (8,-5.5);
\draw[thick,rounded corners=21pt,blue,dashed] (4,-5.5) -- (5,-5.5) -- (6,-7.5) -- (5,-9) -- (0.5,-9) -- (-0.5,-7.5) -- (0.5,-5.5) -- (4,-5.5) ;
\end{scope}

\filldraw[black] (7,-1) circle (2pt) node[anchor=west] {$a^4_1$};
\filldraw[black] (9,-1) circle (2pt) node[anchor=west] {$a^5_1$};

\filldraw[black] (9,-2.5) circle (2pt) node[anchor=west] {$a^5_2$};

\begin{scope}[shift={(-0.5,0)}]
\filldraw[black] (1,-1) circle (2pt) node[anchor=west] {$a^1_1$};
\filldraw[black] (3,-1) circle (2pt) node[anchor=west] {$a^2_1$};
\filldraw[black] (5,-1) circle (2pt) node[anchor=west] {$a^3_1$}; 
\filldraw[black] (3,-2.5) circle (2pt) node[anchor=west] {$a^2_2$};
\filldraw[black] (5,-2.5) circle (2pt) node[anchor=west] {$a^3_2$};  
\draw[thick,rounded corners=15pt] (2,-0.25) -- (3.25,-0.25) -- (4,-1) -- (3.25,-1.75) -- (1.25,-1.75) -- (0.5,-1) -- (1.25,-0.25) -- (2,-0.25);

\draw[thick,rounded corners=15pt] (4,-1.5) -- (5.25,-0.25) -- (6,-1) -- (6,-2.5) -- (5.25,-3.25) -- (3.25,-3.25) -- (2.5,-2.5) -- (4,-1.5);
\end{scope}
\draw[thick,rounded corners=15pt] (8,-0.25) -- (9.25,-0.25) -- (10,-1) -- (9.25,-1.75) -- (7.25,-1.75) -- (6.5,-1) -- (7.25,-0.25) -- (8,-0.25);

\draw[thick,rounded corners=15pt] (8.75,-2.5) -- (9.25,-1.75) -- (10,-2.25) -- (9.25,-3.25) -- (8.75,-2.5);

\end{tikzpicture}
\caption{The multiset $A=\{a_1^1 ,a_1^2  ,a_2^2  ,a_1^3 ,a_2^3, a_1^4,a_1^5 ,a_2^5 \}$ and the set partition $\nu$. 
There exists a row partition $\lambda$ (dashed line) such that each part of $\nu$ is contained in one of the parts of $\lambda$. Hence the partition $\nu$ is not strongly-mixing.} \label{fig7}
\end{figure}

We can now express the Leonov--Shiryaev's formula using the notations and notions relevant to the work done in this paper. 

\begin{tw}[Leonov--Shiryaev's formula]
\begin{equation}
\kappa\left( a_1^1 \cdots a_{k_1}^1 , \ldots ,a_1^n \cdots a_{k_n}^n\right) =
\kappa \left( \prod_{j=1}^{k_i}  a_i^j : i \in [n]  \right) =
\sum_{\nu \in \Poo{A} } \kappa_\nu ,
\end{equation}
\noindent
where the sum on the right-hand side is running over all strongly-mixing partitions of a set $A$.
\end{tw}

\begin{ex}
By Leonov--Shiryaev's formula the cumulant $ \kappa \left( a_1^1 \cdot a_2^1 ,  a_1^2\right) $ expresses as follows:
\begin{equation*}
\begin{array}{lcll}
\kappa \left( a_1^1 \cdot a_2^1,a_1^2 \right)&=& +\k{a_1^1 , a_2^2}\cdot\k{a_2^1} +\k{a_2^1 , a_2^2} \cdot \k{a_1^1}  \\
&&  + \k{a_1^1 , a_2^1 , a_1^2}.
\end{array}
\end{equation*}
\end{ex}

\subsection{Analogue of Leonov--Shiryaev's formula.}

Leonov--Shiryaev's formula relates a cumulant of products with some products of cumulants. 
In the situation we are interested in this paper, where the conditional expected value is the identity mapping, 
we can define two types of cumulants. For each of them we have Leonov--Shiryaev's formula. We present now the third formula, which is a mix of those two.

Consider the identity map:
\begin{equation*}
 \mathcal{\left( A, \cdot\right) } \stackrel{\id}{\longrightarrow} \mathcal{\left( A, \ast\right) }
\end{equation*}
\noindent 
between commutative unital algebras $ \left( \mathcal{A}, \cdot\right) $ and $ \left( \mathcal{A}, \ast\right) $. 
Equation \cref{cumulant2} defined cumulants $\kappa$ of the identity mapping. 
Observe that we can also consider the inverse mapping, namely the map:
\begin{equation*}
 \mathcal{\left( A, \ast\right) } \stackrel{\id^{-1}}{\longrightarrow} \mathcal{\left(A, \cdot\right) }.
\end{equation*}
\noindent
This mapping gives us a way to define cumulants (according to \cref{cumulant2}), which we denote by $\kappa^\ast$.

We present below the Leonov--Shiryaev's formula for both mappings mentioned above: 
\begin{equation*}
\kappa \left( \prod_{j=1}^{k_i}  a_i^j : i \in [n]  \right) =
\sum_{\nu \in \Poo{A} } 
\mathop{\mathlarger{\mathlarger{\mathlarger{\mathlarger{\mathlarger{\ast}}}}}}\limits_{b\in \nu}^{}
\kappa_{}^{} \left( a \in b\right) ,
\end{equation*}
\noindent
and
\begin{equation*}
\kappa_{}^\ast \left( 
\mathop{\mathlarger{\mathlarger{\mathlarger{\mathlarger{\mathlarger{\ast}}}}}}\limits_{j=1}^{k_i} a_i^j : i \in [n]  \right) =
\sum_{\nu \in \Poo{A} } 
\prod_{b \in \nu}
\kappa_{}^\ast \left( a \in b\right)  ,
\end{equation*}
\noindent
where the sums in both equalities run over all strongly-mixing partitions of a multiset~$A =\{a_i^j : i\in[n], j\in [k_i] \}$. 
Observe that in each equality the cumulants on each side are of the same type but the multiplications are not. 
In our formula we will mix types of cumulants on both sides but keep the same multiplication.

To present our result we introduce a class of \textit{strongly-mixing forests} $\Phh{A}$.

\begin{de}
\label{SMF}
Let $A_1, \ldots , A_n $ be multisets consisting of elements of $\mathcal{A}$. Consider a reduced mixing forest $F \in \Pho{A}$ consisting of trees $T_1 , \ldots ,T_s$. Denote by $a_v\in A$ the label of a leaf $a\in A$. 
We define a partition ${\nu}_F$ of a set $A$ as follows:
\begin{equation*}
{\nu}_F = \bigg\{  \left\lbrace a_v : v \in T_k \right\rbrace_{k\in [s]}     \bigg\}.
\end{equation*}
\noindent
We say that a mixing reduced forest $F \in \Pht{A}$ is \textit{strongly-mixing} if the partition $\nu_F$ is strongly-mixing partition. We denote the set of such forests by $\Phh{A}$.
\end{de}

\begin{rem}
Observe that the class of all strongly-mixing forests $\Phh{A}$ is a subclass of all mixing forests $\Pho{A}$, which itself is a subclass of all reduced forests $\Pht{A}$, \emph{i.e.}:
\begin{equation*}
\Phh{A} \subset \Pho{A} \subset \Pht{A} .
\end{equation*}
\noindent
This is analogue to the natural order between classes of strongly-mixing partitions $\Poo{A}$, mixing partitions $\Po{A}$ and partitions $\Pp{A}$:
\begin{equation*}
\Poo{A} \subset \Po{A} \subset \Pp{A} .
\end{equation*}
\end{rem}
We can reformulate \cref{pierwszetwierdzenie} as follows.

\begin{tw}[Analogue of Leonov--Shiryaev's formula]
\label{L-S}
Consider an algebra $\mathcal{A}$ with two multiplicative structures $\cdot$ and $\ast$. Denote by $\kappa$ and $\kappa^\ast$ cumulants related to the identity map on $\mathcal{A}$ as we described above. Then the following formula holds:
\begin{equation}
\kappa^\ast \left( \mathop{\mathlarger{\mathlarger{\mathlarger{\mathlarger{\mathlarger{\ast}}}}}}\limits_{j=1}^{k_i}  a_i^j : i \in [n] \right)  =
\sum_{F \in \Phh{A}} \left( -1\right)^{w_F^{}} \kappa_F^{} ,
\end{equation}
\noindent
where $\Phh{A}$ is a set consisting of strongly-mixing reduced forests.
\end{tw}

\begin{ex}
\cref{fig0} presents all reduced forests on the multiset 
$A =\{ a_1^1 , a_2^1  , a_1^2    \}$. 
Six of them are mixing, five of them are strongly mixing. 
Thus: 
\begin{equation*}
\begin{array}{lcll}
\kappa^* \left( a_1^1 \ast a_2^1 , a_1^2\right) &=& a_1^1 \ast a_2^1\ast a_1^2 - \k{a_1^1 , a_1^2}\ast\k{a_2^1}\\
&& - \k{a_2^1 , a_1^2} \ast \k{a_1^1} - \k{a_1^1 , a_2^1 , a_1^2} \\
&&  +\kappa \left( \k{a_2^1 , a_2^2},\k{a_1^1}\right) . 
\end{array}
\end{equation*}
\end{ex}

\begin{proof}
In \cref{cumulant2} we present the moment cumulant formula for cumulants $\kappa$ related to the map 
$ \mathcal{\left( A, \cdot\right) } \stackrel{\id}{\longrightarrow} \mathcal{\left( A, \ast\right) }$. 
Similar expression for cumulants $\kappa^\ast_{}$ related to the inverse map $\mathcal{\left( A, \ast\right) } \stackrel{\id^{-1}}{\longrightarrow} \mathcal{\left(A, \cdot\right) }$ is of the following form:
\begin{equation*}
 a_1 \cdots a_n  =
\sum_{\nu \in \Pp{[n]}} \mathop{\mathlarger{\mathlarger{\mathlarger{\mathlarger{\mathlarger{\ast}}}}}}\limits^{}_{b\in \nu} \kappa^\ast_{} \left( a_i : i \in b \right)  .
\end{equation*}
\noindent
We express the $\cdot$-product $\left( a^1_1 \ast \cdots \ast a_{k_1}^1 \right)  \cdots \left(  a^n_1 \ast \cdots \ast a_{k_n}^n \right)$ 
via the moment cumulant formula given by the equation above:
\begin{equation}
\label{equal}
\left( a^1_1 \ast \cdots \ast a_{k_1}^1 \right)  \cdots \left(  a^n_1 \ast \cdots \ast a_{k_n}^n \right) = \sum_{\nu \in \Pp{[n]}} \mathop{\mathlarger{\mathlarger{\mathlarger{\mathlarger{\mathlarger{\ast}}}}}}\limits^{}_{b\in \nu} \kappa_{}^\ast \left( 
\mathop{\mathlarger{\mathlarger{\mathlarger{\mathlarger{\mathlarger{\ast}}}}}}\limits_{j=1}^{k_i} a_i^j : i \in [n]  \right) .
\end{equation}

From \cref{pierwszetwierdzenie} we can express the left-hand side of this equation in another way:
\begin{equation*}
\left( a^1_1 \ast \cdots \ast a_{k_1}^1 \right)  \cdots \left(  a^n_1 \ast \cdots \ast a_{k_n}^n \right) 
= \sum_{F \in \Pho{A}} \left( -1\right)^{w_F^{}} \kappa_F^{} .
\end{equation*}
\noindent
Observe that we can split the summation of $\left( -1\right)^{w_F^{}} \kappa_F$ over all mixing reduced forests $F\in \Pht{A}$ into $\ast$-product of summation over all strongly-mixing reduced forests:
\begin{equation*}
\sum_{F \in \Pho{A}} \left( -1\right)^{w_F^{}} \kappa_F^{} = 
\sum_{\nu \in \Pp{[n]}} \mathop{\mathlarger{\mathlarger{\mathlarger{\mathlarger{\mathlarger{\ast}}}}}}\limits^{}_{b\in \nu}  
\left( \sum_{F \in \Phh{A^b}} \left( -1\right)^{w_F^{}} \kappa_F^{} \right) ,
\end{equation*}
\noindent
where, for each partition $\nu$, sets $A^b :=\cup_{i\in b} A_i$ are division of a set $A$.

Observe, that quantities:
\begin{equation*}
\sum_{F \in \Phh{A^b}} \left( -1\right)^{w_F^{}} \kappa_F^{}
\end{equation*}
\noindent
satisfy the system of equations given by the moment cumulant formula \cref{equal}, which has a unique solution. 
This yields the statement of the theorem.
\end{proof}

\begin{rem}
The above equation is still valid when we replace $\kappa$ (which is hidden in $\kappa_F$ terms) with $\kappa^\ast$ and replace $\ast$-products with $\cdot$-products simultaneously.
\end{rem}

\subsection{Approximate factorization property.}
\label{AFP}
In many cases cumulants are quantities of a very small degree.  
The following definition specifies this statement \cite[Definition 1.8]{Sniady2016}.

\begin{de}
Let $\mathcal{A}$ and $\mathcal{B}$ be filtered unital algebras and let $\mathbb{E}: \mathcal{A}\longrightarrow \mathcal{B}$ be a unital linear map. Let $\kappa$ be the corresponding cumulants. We say that $\mathbb{E}$ has \emph{approximate factorization property} if for all choices of $a_1, \ldots , a_l \in \mathcal{A}$ we have that
\begin{equation*}
\deg_\mathcal{B} \kappa \left( a_1, \ldots ,a_l \right) \leq \deg_\mathcal{A} a_1 + \cdots + \deg_\mathcal{A} a_l - 2 \left( l-1 \right)  .
\end{equation*}
\end{de}

\begin{ob}
\label{rest}
Let us go back to the case, when $\mathbb{E}$ is the identity map on algebra $\mathcal{A}$ with two multiplications. 
Suppose that the identity map 
\begin{equation*}
 \left( \mathcal{A}, \cdot \right) \stackrel{\id}{\longrightarrow} \left( \mathcal{A}, \ast \right) ,
\end{equation*}
\noindent
satisfies the approximate factorization property. 

Let $A_1, \ldots A_n $ be multisets consisting of elements of $\mathcal{A}$. Let $A$ be the sum of those multisets. Then for any forest $F \in \Pht{A}$ consisting of $f$ trees, there is the following restriction on the degree of cumulants:
\begin{equation*}
\deg \kappa_F^{} \leq \left( \sum_{a \in A} \deg a \right) - 2 |A|  + 2f,
\end{equation*} 
\noindent
where $|A|$ is the number of elements in $A$.
\end{ob}

\begin{proof}
We analyse the definition of $\kappa_F$ (\cref{xxx}). For any vertex $v \in F$ we defined the quantities $\kappa_v$. Using the approximate factorization property, observe that:
\begin{equation*}
\deg \kappa_v \leq \sum_{i=1}^n \deg \kappa_{v_i} - 2 \left( n +1\right) .
\end{equation*}
\noindent
Going from the root $r$ to the leaves we obtain:
\begin{equation*}
\deg \kappa_r \leq \sum_{i=1}^{n_r} \deg \kappa_{v_i} - 2 \left( n_r +1\right)  .
\end{equation*}
\noindent
where $n_r$ is the number of leaves in a tree rooted in $r$, and $v_i$ for $i\in [n_r]$ are leaves of this tree. 

The cumulants $\kappa_F$ were defined as follows:
\begin{equation*}
\begin{array}{l}
\kappa_F^{} := \mathop{\ast}\limits_{i}^{} \kappa_{V_i}^{}  ,
\end{array}
\end{equation*} 
\noindent
where  $V_i$ are roots of all trees in $F$, hence $\deg \kappa_F^{} \leq \sum \deg \kappa_{V_i}$. It is now easy to see the statement of this observation.
\end{proof}

\subsection{Application: Jack characters and their structure constants.}
\label{JC}

Jack characters provide dual information about Jack polynomials which are a ‘simple’ version of Macdonald polynomials \cite{Lapointe1995}. 
Connections of Jack polynomials with various fields of mathematics and physics were established (read more in \cite{Sniady2015}). 
Therefore a better understanding of Jack characters might shed some light on Jack polynomials.
It seems that behind Jack characters stands a combinatorics of maps, \emph{i.e.}~graphs on the surfaces \cite{DFS2013}. 

Jack characters $\Ch_\pi$ form a natural family (indexed by partitions $\pi$) of functions on
the set $\mathbb{Y}$ of Young diagrams. 
One can introduce two different multiplicative structures on the linear space spanned by Jack characters. 

The $\ast$-product is given by concatenations of partitions:
\begin{equation*}
\Ch_\pi \ast \Ch_\sigma = \Ch_{\pi \sqcup \sigma} .
\end{equation*}

For any partitions $\pi$ and $\sigma$ one can uniquely express the
pointwise product of the corresponding Jack characters
\begin{equation*}
\Ch_\pi \cdot \Ch_\sigma =\sum_\mu g_{\pi,\sigma}^\mu \Ch_\mu \left( \delta \right) 
\end{equation*}
in the linear basis of Jack characters. The coefficients $g_{\pi,\sigma}^\mu (\delta )\in \mathbb{Q}[\delta]$
in this expansion are called \emph{structure constants}. 
Each of them is a polynomial in the deformation parameter $\delta$, on which Jack characters depend implicitly. 
The existence of such polynomials was proven in \cite{Dolega2014}. 
There are several combinatorial conjectures about structure coefficients \cite{Sniady2016} and some partial results \cite{MR3484758,Bur16}. 
Structure constants are closely related to \emph{structure coefficients} introduced by Goulden and Jackson in \cite{MR1325917}.

Śniady considers an algebra of Jack characters as a graded algebra, with gradation given by the notion of \emph{$\alpha$-polynomial functions} 
\cite{Sniady2015}. 
Jack characters are $\alpha$-polynomial function of the following degrees
\[\deg \Ch_\pi = |\pi|+\ell(\pi) .\]
Śniady gave explicit formulas for the top-degree homogeneous part of Jack characters.
We sketch shortly how we use the result presented in the this paper in order to find the top-degree coefficients of the structure constants below.

Consider two integer partitions $\pi =(\pi_1 ,\ldots ,\pi_n)$ and $\sigma =(\sigma_1 ,\ldots ,\sigma_l)$ and the relevant multiset $A=A_1 \cup A_2$ given by:
\begin{equation*}
\begin{array}{l}
A_1= \left\lbrace \Ch_{\pi_1} ,\ldots ,\Ch_{\pi_n} \right\rbrace ,\\
A_2= \left\lbrace \Ch_{\sigma_1} ,\ldots ,\Ch_{\sigma_{l}} \right\rbrace .
\end{array}
\end{equation*}
\noindent
Together with the $\cdot$-product and the $\ast$-product described above, the linear space spanned by Jack characters becomes an algebra with two multiplications. 
We can introduce cumulants as a way of measuring the discrepancy between those two types of multiplications via \cref{cumulant2}. 
Recently the approximation factorisation property of cumulants was proven \cite{Sniady2016}.

\begin{lem}[reformulation of the main result]
\label{long}
Let $A_1, \ldots , A_n $ be multisets consisting of elements of $\mathcal{A}$. Let $A$ be the sum of those multisets. Then:
\begin{equation*} 
\left( a^1_1 \ast \cdots \ast a_{k_1}^1 \right)  \cdots \left(  a^n_1 \ast \cdots \ast a_{k_n}^n \right) =
\sum_{\nu \in \mathcal{P}(A)}
\mathop{\mathlarger{\mathlarger{\mathlarger{\mathlarger{\mathlarger{\ast}}}}}}\limits_{i=1}^{| \nu |}
 \sum_{T\in \Pto{\nu_i}} \left( -1\right)^{w_T^{}} \kappa_T^{} .
\end{equation*}
\noindent
where $\ast$ and $\cdot$ are two different multiplications on $\mathcal{A}$ and $\nu = \lbrace \nu_1, \ldots , \nu_{| \nu |} \rbrace$ is a partition of $A$.
\end{lem}

\begin{proof}
\cref{pierwszetwierdzenie} presents $\left( a^1_1 \ast \cdots \ast a_{k_1}^1 \right)  \cdots \left(  a^n_1 \ast \cdots \ast a_{k_n}^n \right)$ as a sum over reduced mixing forests of cumulants associated to those forests. Observe that each reduced mixing forest $F$ splits naturally into a collection of trees $T_1, \ldots, T_k$. 
Each of $T_i$ possesses the property of being reduced and mixing. Leaves of $F$ are labelled by elements of $\mathcal{A}$, thus we denoted by $A$ the multiset consisting of those labels. Division of $F$ into $T_1, \ldots, T_k$ determines a partition $\nu = \lbrace \nu_1, \ldots , \nu_{k} \rbrace$ of a set $A$, namely $\nu_i \subset A$ consists of all labels of leaves of $T_i$. 
The cumulant $\kappa_F$ is equal to:
\begin{equation*}
\mathop{\mathlarger{\mathlarger{\mathlarger{\mathlarger{\mathlarger{\ast}}}}}}\limits_{i=1}^{k}
 \kappa_{T_i}^{} 
\end{equation*}
\noindent
by the definition. Moreover $ \left( -1\right)^{w_{F}^{}} = \prod_{i=1}^k \left( -1\right)^{w_{T_i}^{}} $. 
\end{proof}

\begin{tw}
With notation presented above, for any two partitions $\pi $ and $\sigma$, the following decomposition is valid
\begin{equation}
\label{Struct}
\Ch_\pi \cdot \Ch_\sigma =
\sum_{\nu \in \mathcal{P}(A)}
\mathop{\mathlarger{\mathlarger{\mathlarger{\mathlarger{\mathlarger{\ast}}}}}}\limits_{i=1}^{| \nu |}
 \sum_{T\in \Pto{\nu_i}} \left( -1\right)^{w_T^{}} \kappa_T^{} ,
\end{equation}
\noindent
where $\nu = \lbrace \nu_1, \ldots , \nu_{| \nu |} \rbrace$ is a partition of $A$ and $\Pto{\nu_i}$ denotes the set of all reduced mixing trees on $\nu_i \subseteq A$.

Moreover, there is the following restriction on the degree of products of cumulants:
\begin{equation*}
\deg \Bigg(
\mathop{\mathlarger{\mathlarger{\mathlarger{\mathlarger{\mathlarger{\ast}}}}}}\limits_{i=1}^{| \nu |}
 \sum_{T\in \Pto{\nu_i}} \left( -1\right)^{w_T^{}} \kappa_T^{} 
\Bigg)
\leq  |\pi | + |\sigma | + 2 | \nu |,
\end{equation*}
\noindent
where $ | \nu |$ is the number of parts in partition $\nu$. 
\end{tw}

Presented statement is based on \cref{long}, the bound of a degree follows immediately from \cref{rest}.

The division given in \cref{Struct} is a toll for capturing structure constants $g_{\pi ,\sigma}^\mu$. It opens a way for induction over the number $\ell(\sigma) +\ell(\pi)$. More precisely, we express $\kappa_T$ in in the linear basis of Jack characters inductively, according to the number of leaves. In a forthcoming paper \cite{Bur16} we give an explicit combinatorial interpretation for the coefficients of high-degree monomials in the deformation parameter $\delta$.

\subsection{How to prove the main theorem?}
\cref{pierwszetwierdzenie} is a straightforward conclusion from two propositions which we present in this section. In our opinion they are interesting themselves. 

We begin by introducing a gap-free vertex colouring on forests $F \in \Pht{A}$. 

\begin{de}
\label{def3}
For a reduced forest $F$ with leaves in a multiset $A = A_1 \cup \cdots \cup A_n $ we say that $c$ is a gap-free vertex colouring with length $r$ if 
\begin{itemize}
\item $c$ is a coloured by the numbers $\{ 0,\ldots,r\}$ and each colour is used at least once;
\item each leaf is coloured by $0$; 
\item the colours are strictly increasing on any path from the root to a leaf.
\end{itemize}
\noindent
We denote by $|c|:=r$ the length of $c$ . We call such a colouring $c$ \textit{weakly-mixing} if it satisfies one of the following additional conditions: 
\begin{enumerate}
\item either there exists a vertex coloured by $1$ with at least two descendants, each of whom belongs to a distinct multiset $A_i$,
\item or colouring $c$ does not use the colour $1$ at all.
\end{enumerate}
\noindent
We denote by $\mathcal{C}_F$ the set of all gap-free and weakly-mixing colourings of a forest $F$.
\end{de}

The following result is a juggling of a concept of cumulants. We present its proof in \cref{5}.

\begin{restatable}{prop}{przedostatni}
\label{przedostatnie1}
Let $A_1, \ldots , A_n $ be multisets consisting of elements of $\mathcal{A}$. Let $A$ be a sum of those multisets. Then
\begin{equation}
\label{zzz}
\left( a^1_1 \ast \cdots \ast a_{k_1}^1 \right)  \cdots \left(  a^n_1 \ast \cdots \ast a_{k_n}^n \right) 
= \sum_{F \in \Pht{A} } \kappa_F^{} \sum_{\ c \in \mathcal{C}_F } {(-1)}^{|c|} .
\end{equation}
\end{restatable}

In \cref{6}, we will show that summing over all colourings $c \in \mathcal{C}_F$ for a reduced forest $F \in \Pht{A}$ gives a surprisingly simple number. This result is presented in proposition below. 

\begin{restatable}{prop}{proppp}
\label{przedostatnie2}
Let $A_1, \ldots , A_n $ be multisets consisting of elements of $\mathcal{A}$. Let $A$ be a sum of those multisets. Then, for any reduced forest $F \in \Pht{A}$, the following holds:
\begin{equation*}
\sum_{c \in \mathcal{C}_F} {\left( -1\right) }^{|c| } = 
\left\lbrace 
\begin{array}{lccl}
{\left( -1\right) }^{w_F }  &&& $if $ F \in \Pho{A} ,\\
0                           &&& $otherwise$.
\end{array}
\right. 
\end{equation*}

\end{restatable}

Observe that combing \cref{przedostatnie1} and \cref{przedostatnie2} we obtain the statement of Theorem \ref{pierwszetwierdzenie}.

\subsection{Related work and free probability theory.}
A \emph{noncommutative probability space} is a pair $\left( \mathcal{A}, \phi \right) $ consisting of a unital algebra $\mathcal{A}$ and a linear form $\phi$ on $\mathcal{A}$ such that $\phi (1) =1$, which is called a \emph{noncommutative expectation} \cite{Speicher94,Speicher98}. Functions
\begin{equation*}
m_n \left( a_1,\ldots,a_n\right)  := \phi  \left( a_1\cdots a_n\right) 
\end{equation*} 
\noindent
are called \emph{free moments}.

Roland Speicher introduced the free cumulant functional \cite{Speicher94} in the free probability theory. It is related to the lattice of noncrossing partitions of the set $[n]$ in the same way in which the classic cumulant functional is related to the lattice of all partitions of that set.

\begin{de}
A partition $\nu \in \mathcal{P}([n])$ is \emph{noncrossing} if there is no quadruple of elements $i<j<k<l$ such that $i\sim_{\nu} k$,  
$j\sim_{\nu} l$, and  $\neg ( i\sim_{\nu} j )$, where ''$\sim_{\nu}$'' denotes the relation of being in the same set in the partition $\nu$.
\end{de}

The \emph{free cumulants} are defined implicit by the system of equations
\begin{equation*}
\phi \left( a_1 \cdots a_n \right) =
\sum_{\nu \in \mathcal{P}_{\mathbb{NC}} ([n])}\prod_{b \in \nu} \kappa\left( a_i : i \in b \right) 
\end{equation*}
\noindent
where the sum runs over all \text{noncrossing} set partitions of $[n]$, compare to \cref{cum2}. Möbius inversion over the lattice of noncrossing partitions gives the formula for the free cumulants in terms of the free moments. 

Josuat-Verg\`es, Menous, Novelli and Thibon \cite[Theorem 4.2]{MR3641810} give
 formulas for free cumulants in terms of Schröder trees, \emph{i.e.}~reduced plane trees for which the rightmost sub-tree is a leaf. To each such a tree they associate the term constructed by the mapping $\phi$. 
At the first glance the formula seems to be related to the formula we give in a current paper. 
However, the reasons for appearance of reduced trees in both papers are different. 
In their work reducedness of trees is a natural property appearing while recovering the \emph{free cumulant} from free moments. 
In our case, where $\mathbb{E}$ or $\phi$ is the identity, we have
 $\kappa \left( a\right) =a$ for any element $a\in \mathcal{A}$. Hence we consider reduced trees. 
The notion of free cumulants is based on noncrossing partitions which follows flatness of trees they consider.\hfill \break

There are approaches to freeness 
other then 
considering a linear form $\phi$ on an algebra $\mathcal{A}$. The standard and the most common approach is to additionally require that $\mathcal{A}$ is a $\mathcal{B}$-module or $\mathcal{B} \subseteq \mathcal{A}$ is a subalgebra of $\mathcal{A}$. The mapping $\phi  : \mathcal{A}\longrightarrow \mathcal{B}$ satisfies the bimodule map property:  
\begin{equation*}
\phi\Big(   b_1  \cdot a \cdot b_2  \Big) = b_1 \cdot \phi \left( a\right)  \cdot b_2 ,
\end{equation*}
\noindent
for any $a \in \mathcal{A}$ and $b_1, b_2 \in \mathcal{B}$. 
There are several slightly different approaches, \emph{e.g.}, free products came with amalgamation over module $\mathcal{B}$ \cite{Voiculescu}, where cumulants and moments are operator-valued multiplicative functions \cite[Definition 2.1.1]{Speicher98}. 

Our work is based on the idea of taking an identity map between elements of an algebra with two different multiplications ($\phi \equiv \id$). Such situation does not arise naturally in free probability theory, where $\phi$ is usually either linear form or bimodule map on an algebra $\mathcal{A}$. It rises the question if it is still possible to define naturally cumulants in the setup of noncommutative algebras with two different products.

\section{Proof of \cref{przedostatnie1}}
\label{5}

In this section we shall prove \cref{przedostatnie1}.
We use the same notation as in \cref{main subsection}. 
We denote by $A_1 , \ldots ,A_n$ multisets consisting of elements of $\mathcal{A}$. 
We denote by $ A = A_1 \cup \cdots \cup A_n$ the multiset, which is the sum of all multisets $A_i$.
We use also the following notation for the elements of $A_i$: 
\begin{equation*}
A_i =\left\lbrace  a_1^i ,\ldots ,a_{k_i}^i \right\rbrace ,
\end{equation*}
\noindent
hence the multiset $A$ consists of the following elements: 
\begin{equation*}
A =\left\lbrace  a_1^1 ,\ldots ,a_{k_1}^1 , \ldots ,  a_1^n ,\ldots ,a_{k_n}^n \right\rbrace .
\end{equation*}
\noindent
We denote additionally the set of all partitions of $A$ by $\Pp{A} $. We denote by $\Po{A} $ a set of all \textit{mixing partitions} of $A$, i.e. all partitions $\nu =\{ \nu_1 ,\ldots, \nu_l \}$ such that 
\begin{equation*}
 \mathop{\exists}\limits_{i \in [l]}^{}    
 \mathop{\forall}\limits_{j \in [n]}^{}
\nu_i \not\subseteq A_j .
\end{equation*}

\subsection{Outline of the proof.}
Firstly, we express the left-hand side of \cref{zzz} 
as a sum of cumulants, where the sum runs over all mixing partitions 
$\nu \in  \Po{A}$, see \cref{sumakumulant} below. By applying inductively the procedure \eqref{procedure} described below, we replace summation over all mixing partitions $\nu \in  \Po{A}$ by a sum over all \textit{nested upward sequences of partitions}, see \cref{def1}. Then we construct a bijection between such sequences and reduced forests $F \in \Pht{A}$ equipped with gap-free, weakly-mixing colourings $c\in \mathcal{C}_F$ (see Definitions \ref{root}, \ref{def3}). 
Later on we will prove that the weighted sum over all gap-free colourings for a fixed forest is either equal to $0$ or to $\pm 1$.

\subsection{Cumulants of mixing partitions.}

Observe that the following equality of the sets holds:
\begin{equation*}
\Pp{A} = \Big( \Pp{A_1} \times \cdots \times\Pp{A_n} \Big)  \ \cup \  \Po{A} ,
\end{equation*}
\noindent
where the elements of the Cartesian product $\Pp{A_1} \times\cdots \times \Pp{A_n}$ are understood as a partition of a multiset $A =A_1 \cup \cdots \cup A_n$.

We apply the moment-cumulant formula given in \cref{cumulant2}:

\begin{equation*}
 a^1_1 \ast \cdots \ast a_{k_1}^1  \ast\cdots\ast   a^n_1 \ast \cdots \ast a_{k_n}^n 
                  =\sum_{\nu \in \Pp{A} } \kappa_\nu .
\end{equation*}

\noindent
We split all partitions $\Pp{A}$ into two categories: mixing partitions $\Po{A}$ and products of partitions $\Pp{A_i}$. In this way:
\begin{equation*}
\begin{split}
\sum_{\nu \in \Pp{A} } \kappa_\nu 
=\sum_{\nu \in \coprod_{i\in [n]} \Pp{A_i} } \kappa_\nu   + \sum_{\nu \in \Po{A} } \kappa_\nu   
                   =\prod_{i=1}^n \sum_{\nu \in \Pp{A_i} } \kappa_\nu   + \sum_{\nu \in \Po{A} } \kappa_\nu   \\
                         =\left( a^1_1 \ast \cdots \ast a_{k_1}^1 \right)  \cdots \left(  a^n_1 \ast \cdots \ast a_{k_n}^n \right)+  \sum_{\nu \in \Po{A} } \kappa_\nu .
\end{split}
\end{equation*}

\noindent
From the equations above we obtain the following formula:
\begin{equation}
\label{sumakumulant}
\begin{split}
\left( a^1_1 \ast \cdots \ast a_{k_1}^1 \right) \cdots &\left(  a^n_1 \ast \cdots \ast a_{k_n}^n \right) \\ &= \left( a^1_1 \ast \cdots \ast a_{k_1}^1 \right)  \ast\cdots\ast \left(  a^n_1 \ast \cdots \ast a_{k_n}^n \right)
-  \sum_{\nu \in \Po{A} } \kappa_\nu .
\end{split}
\end{equation}

\subsection{Cumulants of upward sequences of partitions.}

Each cumulant on the right-hand side of \cref{sumakumulant} is a $\cdot$-product of simple cumulants. 
We use the moment-cumulant formula in a form given below
\begin{equation}
\label{procedure}
a_1 \cdots a_k =a_1\ast \ldots\ast a_k -\sum_{\substack{\nu  \in \Pp{[k]} \\ \nu \neq \{\{ 1\},\ldots ,\{ k\}\}}} \c{a_1,\ldots, a_k} .
\end{equation}
\noindent
to replace $\cdot$-products by $\ast$-products and $\cdot$-products consisting of a strictly smaller number of components.

For each cumulant on the right-hand side in \cref{sumakumulant} we apply the procedure \cref{procedure}. As an output we get one term which is a $\ast$-product of cumulants and several terms of the form of a $\cdot$-product of cumulants. Observe that in each term of the second type the number of factors is strictly smaller than before applying the procedure. We apply to them this procedure iteratively as long as we have $\cdot$-terms in our extension. In the end we get a sum of the terms given by $\ast$-product and cumulants. 

\begin{ex}
Let us express $\left( a_1^1 \ast a_2^2 \right) \cdot a_1^2$ using the procedure described above:
\begin{equation*}
\begin{array}{lcll}
\left( a_1^1 \ast a_2^1 \right) \cdot a_1^2&\overset{\eqref{sumakumulant}}{=}& a_1^1 \ast a_2^1\ast a_1^2 - \k{a_1^1 , a_2^1 , a_1^2} \\
&& - \k{a_1^1 , a_1^2}\cdot\k{a^1_2} -   \k{a_2^1 , a_1^2} \cdot \k{a_1^1}   \\
\\
&\overset{\eqref{procedure}}{=}& a_1^1 \ast a_2^1\ast a_1^2 - \k{a_1^1 , a_2^1 , a_1^2}\\
&& -\kappa \left( \k{a_1^1 , a_1^2},\k{a_2^1}\right) +\kappa \left( \k{a_2^1 , a_1^2},\k{a_1^1}\right)  \\
&& - \k{a_1^1 , a_1^2}\ast\k{a_2^1} + \k{a_2^1 , a_1^2} \ast \k{a_1^1} .
\end{array}
\end{equation*}
\end{ex}

To formalize our idea we define nested upward sequences and theirs cumulants.

\begin{de}
\label{def1}
A sequence of partitions $\omega = \left( \nu^1 \nearrow \cdots \nearrow \nu^r \right) $ is said to be \textit{upward} if 
\begin{equation*}
\nu^{i+1} \textrm{ is a partition of the set } \nu^{i} ,
\end{equation*}
for any $1\leq i \leq r-1$ and $\nu^1$ is a partition of a multiset $A$. Moreover, if for each $i$ the partition $\nu^{i+1}$ is non-trivial,
\emph{i.e.} $\nu^{i+1} \neq \left\lbrace  \nu^{i} \right\rbrace $, it is said to be \textit{nested}. 
We define the length of an upward sequence of partitions $\omega =\left( \nu^1 \nearrow \cdots \nearrow \nu^r\right) $ as the length of a sequence, and we denote $|\omega | =r$.
\end{de}

Let us provide a simple example.

\begin{ex}
\label{ex}
Consider a $5$-element multiset $A= \lbrace a_1, \ldots, a_5 \rbrace$ and the following nested upward sequences of partitions $\omega_1 = (\nu^1 \nearrow \nu^2 )$ and $\omega_2 = (\nu^1 \nearrow \nu^2 \nearrow \nu^3 )$, where:
\begin{equation*}
\begin{array}{l}
\nu^1 = \Big\{ \big\{a_1 ,a_4 \big\}, \big\{a_2 \big\}, \big\{a_3\big\}, \big\{a_5\big\} \Big\} ,\\
\\
\nu^2 = \bigg\{ \Big\{ \big\{a_1 ,a_4 \big\}, \big\{a_2 \big\}, \big\{a_3 \big\}\Big\}, \Big\{\big\{a_5 \big\} \Big\} \bigg\}, \\
\\
\nu^3 = \Bigg\{ \bigg\{ \Big\{ \big\{a_1 ,a_4 \big\}, \big\{a_2 \big\},  \big\{a_3 \big\} \Big\}, \Big\{\big\{a_5 \big\} \Big\} \bigg\} \Bigg\} .
\end{array}
\end{equation*}
\end{ex}

We introduce the following technical notation (similar to the definition of a cumulant $\kappa_\nu$ for a partition $\nu$). 
\begin{equation*}
\overline{\kappa}_\nu^{} \lbrace a_1, \ldots, a_n \rbrace :=  \big\{ \k{a_{i_1} , \ldots, a_{i_{|\nu_j|}}} {\big\}}_{j=1}^l .
\end{equation*}

\begin{de}
\label{cumm}
Let $\nu$ be a partition of a multiset $A =\{a_1 , \ldots ,a_n \}$. Consider an upward sequence of partitions $\omega = \left( \nu^1 \nearrow \cdots \nearrow \nu^r \right) $ such that 
\begin{equation*}
\nu^{1} =\nu .
\end{equation*}
\noindent
We define the cumulant associated to the sequence $\omega$ as follows
\begin{equation*}
\kappa_\omega^{} := 
\mathop{\mathlarger{\mathlarger{\mathlarger{\mathlarger{\mathlarger{\ast}}}}}}\limits^{}_{b \enskip\in\enskip
{\overline{\kappa}}_{\nu^{r-1}} \Big( \ldots  {\overline{\kappa}}_{\nu^1} \big( a_1, \cdots ,a_n  \big) \Big)  } b.
\end{equation*}
\end{de}

\begin{ex}
The cumulants $\kappa_{\omega_1}$ and $\kappa_{\omega_2}$ associated to the nested upward sequences of partitions $\omega_1$ and $\omega_2$ respectively  from \cref{ex} are of the following forms:
\begin{equation*}
\begin{array}{ll}
\kappa_{\omega_1}^{} &=\kappa \Big( \k{a_1, a_4},\k{a_2},\k{a_3} \Big)\ast \kappa \Big(\k{a_5}\Big) \\
&=\kappa \Big(\k{a_1, a_4},a_2,a_3 \Big)\ast a_5 , \\
\\
\kappa_{\omega_2}^{} &=\kappa \Bigg( \kappa \Big( \k{a_1, a_4},\k{a_2},\k{a_3} \Big) , \kappa \Big(\k{a_5}\Big) \Bigg)\\
&=\kappa \Bigg( \kappa \Big(\k{a_1, a_4},a_2,a_3 \Big), a_5 \Bigg), \\
\end{array}
\end{equation*}
\noindent
where we used the property $\kappa (x) =x$.
\end{ex}

\begin{de}
Consider a multiset $A = A_1 \cup \cdots \cup A_n $. Denote by $ \Pt{A}$ the set of all nested upward sequences of partitions $\omega = \left( \nu^1 \nearrow \cdots \nearrow \nu^r\right) $ such that $\nu^1 \in \Po{A}$ is a mixing partition.
\end{de}

\begin{prop}
Consider a multiset $A=A_1 \cup\cdots\cup A_n$. Then
\label{prop2}
\label{2}
\begin{equation}
\sum_{\nu \in \Po{A} } \kappa_\nu = 
-\sum_{\omega \in \Pt{A} } {(-1)}^{|\omega |}   \kappa_\omega^{} .
\end{equation}
\end{prop}

\begin{proof}
Apply procedure \eqref{procedure} iteratively to the left-hand side of \cref{2}.
Observe that applying this iterative procedure 
is nothing else but summing over all nested upward sequences of partitions $\omega = \left( \nu^1 \nearrow \cdots \nearrow \nu^r\right) $. 
The sign of the term is determined by the number of iterations. 
Partition $\nu_1$ describes the first application of the procedure (this is why $\nu^1 \in \Po{A}$), partition $\nu_2$ the second, and so on. 
\end{proof}

Observe that different nested upward sequences of partitions $\omega$ may lead to the same cumulant $\kappa_\omega$. The following example illustrates this phenomenon. 

\begin{ex}
\label{3omega}
Let $A_1 = \{ a_1^1 ,a_2^1 \}$ and $A_2 =\{a_1^2 ,a_2^2 \}$. Consider $\omega_1, \omega_2, \omega_3 \in  \Pt{A}$ given by 
$\omega_1 =\left(  \nu^1_1 \nearrow \nu^2_1 \right) $ and 
$\omega_2 = \left( \nu^1_2 \nearrow \nu^2_2\nearrow \nu^3_2 \right) $ and 
$\omega_3 = \left( \nu^1_3 \nearrow \nu^2_3\nearrow \nu^3_3 \right) $, where: 
\begin{equation*}
\begin{array}{ll}
\nu^1_1 = \Big\{ \big\{a_1^1 ,a_2^2 \big\}, \big\{a_2^1,a_1^2  \big\}\Big\} ,&
\nu^2_1 = \bigg\{ \Big\{ \big\{a_1^1 ,a_2^2 \big\}, \big\{a_2^1 ,a_1^2 \big\} \Big\} \bigg\}, \\
\\
\nu^1_2 = \Big\{ \big\{a_1^1 ,a_2^2 \big\}, \big\{a_2^1\big\},\big\{a_1^2  \big\}\Big\}, &
\nu^2_2 = \bigg\{ \Big\{ \big\{a_1^1 ,a_2^2 \big\} \Big\},\Big\{ \big\{a_2^1\big\},\big\{a_1^2  \big\} \Big\} \bigg\}, \\
&\nu^3_2 =  \Bigg\{\bigg\{ \Big\{ \big\{a_1^1 ,a_2^2 \big\} \Big\},\Big\{ \big\{a_2^1\big\},\big\{a_1^2  \big\} \Big\} \bigg\} \Bigg\},\\
\\
\nu^1_3 = \Big\{ \big\{a_1^1\big\},\big\{a_2^2 \big\},\big\{a_2^1 ,a_1^2  \big\}\Big\} ,&
\nu^2_3 = \bigg\{ \Big\{ \big\{a_1^1\big\},\big\{a_2^2 \big\} \Big\},\Big\{ \big\{a_2^1 ,a_1^2  \big\} \Big\} \bigg\} ,\\
&\nu^3_3 =  \Bigg\{\bigg\{ \Big\{ \big\{a_1^1 \big\},\big\{a_2^2 \big\} \Big\},\Big\{ \big\{a_2^1 ,a_1^2  \big\} \Big\} \bigg\} \Bigg\} .\\
\end{array}
\end{equation*}
\noindent
Observe that all sequences $\omega_1, \omega_2, \omega_3$ lead to the same term 
$\kappa\Big( \k{a_1^1 ,a_2^2},\k{a_2^1 ,a_1^2 } \Big)$ 
up to the sign. Moreover, they are the only ones which lead to this cumulant. Observe that

\begin{equation*}
\begin{array}{ll}
{(-1)}^{|\omega_1 |}  \kappa_{\omega_1} +
{(-1)}^{|\omega_2 |}  \kappa_{\omega_2} +
{(-1)}^{|\omega_3 |}  \kappa_{\omega_3} =
\kappa\Big( \k{a_1^1 ,a_2^2},\k{a_2^1 ,a_1^2 } \Big) .
\end{array}
\end{equation*} 

\noindent
With the weights given by $ {(-1)}^{|\omega_i |} $, cumulants corresponding to sequences $\omega_1, \omega_2, \omega_3$ sum up to just one term. We will see that this is true in general. 
\end{ex}

\subsection{Reduced forests and their colourings.}
\label{treesproper}

To each upward nested sequence of partitions $\omega =\left( \nu^1\nearrow\cdots \nearrow \nu^r\right) $ we shall assign a certain rooted forest with a colouring. We construct a bijection between the sequences from $\Pt{A}$ and relevant rooted forests equipped with the colourings.

\begin{de}
\label{tree}
Let $\omega =\left( \nu^1 \nearrow\cdots\nearrow\nu^r\right) $ be a nested sequence of partitions. Denote the elements of partition $\nu^i$ by $\nu^i = \{ \nu^i_1,\ldots,\nu^i_{k_i} \}$. Let $\nu^1 =\{ \nu^1_1,\ldots , \nu^1_{k-1} \}$ be a partition of $A =A_1 \cup\cdots\cup A_n$. 
We associate to $\omega$ a rooted forest with coloured vertices by the following procedure:
\begin{itemize}
\item The elements of $ A$ are leaves of the forest. We colour each of them by $0$. 
\item For each element $\nu^i_j$, where $1\leq i \leq r$ and $1\leq j \leq k_i$, we create a vertex and colour it by $i$. 
\item We join $\nu^{i-1}_j$ and $\nu^i_j$ if $\nu^{i-1}_j \subseteq\nu^i_j$. Similarly we join $a\in A $ and $\nu^1_j$ if $a \in\nu^1_j$.
\item We delete each vertex $v$ which has only one descendant. We join the descendant and the parent of $v$.
\end{itemize}
\noindent
We denote by $\Phi_1(\omega )$ the forest and by $\Phi_2(\omega )$ the colouring associated to $\omega$.
\end{de}

\begin{ex}
The coloured forests associated with $\omega_1, \omega_2, \omega_3 $ from Ex.~\ref{3omega} are presented on \cref{fig1}. Since $\omega_1, \omega_2, \omega_3 $ start from one element partition, all three forests are trees.
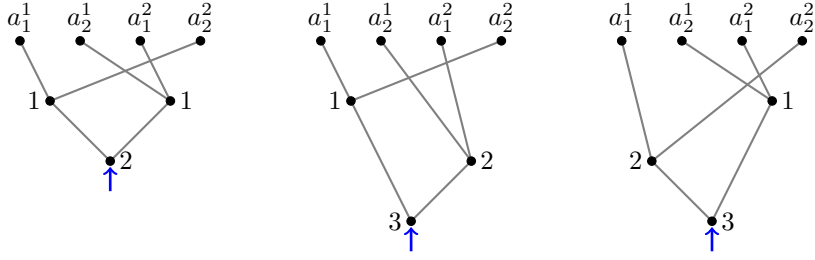
\begin{figure}
\centering
\begin{tikzpicture}[scale=0.8]

\draw[gray, thick] (0,4) -- (0.5,3);
\draw[gray, thick] (3,4) --  (0.5,3);
\draw[gray, thick] (2,4) -- (2.5,3);
\draw[gray, thick] (1,4) -- (2.5,3);

\draw[gray, thick] (0.5,3) -- (1.5,2);
\draw[gray, thick] (2.5,3) -- (1.5,2);

\filldraw[black] (0,4) circle (2pt) node[anchor=south] {$a_1^1$};
\filldraw[black] (1,4) circle (2pt) node[anchor=south] {$a_2^1$};
\filldraw[black] (2,4) circle (2pt) node[anchor=south] {$a_1^2$};
\filldraw[black] (3,4) circle (2pt) node[anchor=south] {$a_2^2$};

\filldraw[black] (0.5,3) circle (2pt) node[anchor=east] {$1$};
\filldraw[black] (2.5,3) circle (2pt) node[anchor=west] {$1$};
\filldraw[black] (1.5,2) circle (2pt) node[anchor=west] {$2$};
\draw[->,line width=1pt, blue] (1.5,1.5) to (1.5,1.9);

\draw[gray, thick] (5,4) -- (5.5,3);
\draw[gray, thick] (8,4) --  (5.5,3);
\draw[gray, thick] (7,4) -- (7.5,2);
\draw[gray, thick] (6,4) -- (7.5,2);

\draw[gray, thick] (5.5,3) -- (6.5,1);
\draw[gray, thick] (7.5,2) -- (6.5,1);

\filldraw[black] (5,4) circle (2pt) node[anchor=south] {$a_1^1$};
\filldraw[black] (6,4) circle (2pt) node[anchor=south] {$a_2^1$};
\filldraw[black] (7,4) circle (2pt) node[anchor=south] {$a_1^2$};
\filldraw[black] (8,4) circle (2pt) node[anchor=south] {$a_2^2$};

\filldraw[black] (5.5,3) circle (2pt) node[anchor=east] {$1$};
\filldraw[black] (7.5,2) circle (2pt) node[anchor=west] {$2$};
\filldraw[black] (6.5,1) circle (2pt) node[anchor=east] {$3$};
\draw[->,line width=1pt, blue] (6.5,0.5) to (6.5,0.9);

\draw[gray, thick] (10,4) -- (10.5,2);
\draw[gray, thick] (13,4) --  (10.5,2);
\draw[gray, thick] (12,4) -- (12.5,3);
\draw[gray, thick] (11,4) -- (12.5,3);

\draw[gray, thick] (10.5,2) -- (11.5,1);
\draw[gray, thick] (12.5,3) -- (11.5,1);

\filldraw[black] (10,4) circle (2pt) node[anchor=south] {$a_1^1$};
\filldraw[black] (11,4) circle (2pt) node[anchor=south] {$a_2^1$};
\filldraw[black] (12,4) circle (2pt) node[anchor=south] {$a_1^2$};
\filldraw[black] (13,4) circle (2pt) node[anchor=south] {$a_2^2$};

\filldraw[black] (10.5,2) circle (2pt) node[anchor=east] {$2$};
\filldraw[black] (12.5,3) circle (2pt) node[anchor=west] {$1$};
\filldraw[black] (11.5,1) circle (2pt) node[anchor=west] {$3$};
\draw[->,line width=1pt, blue] (11.5,0.5) to (11.5,0.9);

\end{tikzpicture}
\caption{The forests associated with $\omega_1, \omega_2, \omega_3 $ from \cref{3omega}. All leaves are coloured by $0$.} \label{fig1}
\end{figure}
\end{ex}

The forest described in Definition \ref{tree} consists of $k_r$ rooted trees, where $k_r$ is a number of elements in $\nu^r$, namely  
$\nu^r = \{ \nu^r_1,\ldots,\nu^r_{k_r} \}$.
The condition of nestedness of $\omega$ translates to the fact that each colour is used. Except for leaves, each vertex has at least two descendants. It leads to the definition of reduced forest and gap-free, weakly-mixing colouring. We mentioned their definitions in the introduction (see \cref{root,def3}).

\begin{lem}
\label{bijekcja}
There exists a bijection $\Phi$ between the set $\Pt{A}$ of nested upward sequences starting with a mixing partition and the set of pairs $(F,c)$ consisting of a reduced forest $F\in\Pht{A}$ of length $r \geq1$ with a gap-free, weakly-mixing colouring $c \in \mathcal{C}_F$:
\begin{equation*}
\Phi : \omega \longmapsto \left( F, c \right):= \left( \Phi_1 (\omega ), \Phi_2 (\omega ) \right) 
\end{equation*} 

For any nested upward sequence starting with a mixing partition $\omega \in \Pt{A}$, the following equality of cumulants holds
\begin{equation*}
\kappa^{}_{\omega}= \kappa^{}_{\Phi_1 (\omega )} ,
\end{equation*}
\noindent
where $\kappa^{}_{\Phi_1 (\omega )}$ is a cumulant of a reduced forest $\Phi_1 (\omega )$, see Definition \ref{xxx}. 

Moreover $|\omega|=|\Phi_2(\omega ) |$ \emph{i.e.}~the length of the nested upward sequence is equal to the length of the corresponding colouring. 
\end{lem}

\begin{proof}
Definition \ref{tree} shows already how to associate a reduced forest $F:= \Phi_1 (\omega )$ with the gap-free colouring $c:=\Phi_2 (\omega )$ to a nested upward sequence $\omega$. The construction is done in such a way that $|c|=|\omega |$. For the reverse direction, the algorithm is easily reproducible. 
The condition that a nested upward sequence $\omega = \left( \nu^1 \nearrow \cdots\nearrow\nu^r \right)  \in \Pt{A}$ starts with a mixing partition $\nu^1 \in \Po{A}$ translates to the condition of $c$ being a weakly-mixing colouring (Definition \ref{def3}). 

In Definition \ref{xxx} we introduced cumulant $\kappa^{}_{F} $ for a forest $F \in \Pht{A}$. There is an exact correspondence between this expression and the one, which is given in Definition \ref{cumm}.
\end{proof}

We are ready to prove Proposition \ref{przedostatnie1} which is the purpose of this section. Let us recall its statement:

\przedostatni*

\begin{proof}
Combining the formula \eqref{sumakumulant} and the Proposition \ref{prop2} lead to the following expression:
\begin{equation*}
\begin{split}
\left( a^1_1 \ast \cdots \ast a_{k_1}^1 \right) \cdots&\left(  a^n_1 \ast \cdots \ast a_{k_n}^n \right) \\ &= \left( a^1_1 \ast \cdots \ast a_{k_1}^1 \right)  \ast\cdots\ast \left(  a^n_1 \ast \cdots \ast a_{k_n}^n \right)
+\sum_{\omega \in \Pt{A} } {(-1)}^{|\omega |}   \kappa_\omega^{} .
\end{split}
\end{equation*}

We identify the product term on the right-hand side of the equation above with the only reduced forest of length $r =0$. Indeed, there is just one reduced forest of length $r =0$ and the only one gap-free, weakly-mixing vertex colouring $c$ of it, namely the forest $F$ consisting of separated vertices $a\in A$, each coloured by $0$. The term $\left( a^1_1 \ast \cdots \ast a_{k_1}^1 \right)  \ast\cdots\ast \left(  a^n_1 \ast \cdots \ast a_{k_n}^n \right)$ is equal to the corresponding cumulant~$\kappa_F$.

We replace the sum term on the right-hand side of the equation above, according to the bijection between sequences $\omega \in \Pt{A}$ and reduced forests of length $r\geq 1$ with gap-free, weakly-mixing colourings given in Lemma \ref{bijekcja}.  
\end{proof}

\section{Proof of Proposition \ref{przedostatnie2}}
\label{6}

In this section we shall prove \cref{przedostatnie2}. For a given reduced forest~$F$, we investigate the following sum
\begin{equation*}
\sum_{ c \in \mathcal{C}_F } {\left( -1\right) }^{|c| } 
\end{equation*} 
\noindent
over all gap-free, weakly-mixing colourings of~$F$, which occur in \cref{przedostatnie1}.

\subsection{Parameter $w_F$ of a reduced forest $F\in \Pht{A}$.}

We introduce an invariant $w_F$ which determines the coefficient of $\kappa_F$. This definition was already mentioned in \cref{main subsection}, we recall it below and next extend it slightly: 

\definitionw*

If $F$ is not mixing, we define $w_F := \infty$. We may also introduce number $w_F$ inductively, according to the height of a forest. 

\begin{de}
Let $T$ be a reduced tree. The height of a tree $T$ is the maximum distance between its root and one of its leaves. The hight of a forest $F$ is a hight of the highest tree in $F$. We denote this quantity by $h(F)$.
\end{de}

\begin{de}[Definition equivalent to \cref{kT}]
\label{remark}
Let $F$ be a reduced forest and $F_1, \ldots,F_r$ its sub-forests obtained by deleting the roots of $F$. We define the number $w_F \in \mathbb{N} \cup \{\infty\}$ inductively on $h(F)$ as follows: 
\begin{equation*}
w_F = 
\begin{cases}
\sum_{i=1}^r k_{F_i} +1   &  \quad $      if $h(F) \geq 2, \\
\infty                    & \quad $      if $h(F) =1$ and all descendants of some root$ \\
                          &   $belong to just one multiset $A_i$ for some $i\in [n], \\
1                         &  \quad $      if $h(F) =1$ and for each root there are at least$  \\
                          &  $two descendants belonging to some two distinct$    \\
                          &  $multisets $A_i$ and $A_j$,$    \\
0                         &  \quad $      if $h(F)=0. \\
\end{cases}
\end{equation*}
\end{de}

\begin{ex} 
Let $A=A_1 \cup A_2 \cup A_3$. On \cref{fig2}, we give an example of two forests (in particular trees) and we count the two corresponding $w_F$ numbers. Observe that number $w_F$ depends on the labels of the leaves of $F$.
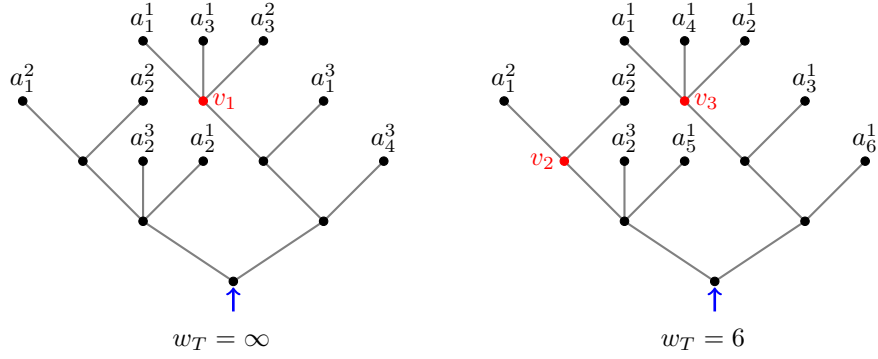
\begin{figure}
\centering
\begin{tikzpicture}[scale=0.8]

\draw[gray, thick] (2,10) -- (3,9);
\draw[gray, thick] (3,10) --  (3,9);
\draw[gray, thick] (4,10) -- (3,9);

\draw[gray, thick] (0,9) -- (1,8);
\draw[gray, thick] (2,9) --  (1,8);
\draw[gray, thick] (5,9) --  (4,8);
\draw[gray, thick] (3,9) -- (4,8);

\draw[gray, thick] (1,8) -- (2,7);
\draw[gray, thick] (2,8) --  (2,7);
\draw[gray, thick] (3,8) -- (2,7);
\draw[gray, thick] (4,8) --  (5,7);
\draw[gray, thick] (6,8) -- (5,7);

\draw[gray, thick] (5,7) --  (3.5,6);
\draw[gray, thick] (2,7) -- (3.5,6);

\filldraw[black] (2,10) circle (2pt) node[anchor=south] {$a_1^1$};
\filldraw[black] (3,10) circle (2pt) node[anchor=south] {$a_3^1$};
\filldraw[black] (4,10) circle (2pt) node[anchor=south] {$a_3^2$};

\filldraw[black] (0,9) circle (2pt) node[anchor=south] {$a_1^2$};
\filldraw[black] (2,9) circle (2pt) node[anchor=south] {$a_2^2$};
\filldraw[black] (5,9) circle (2pt) node[anchor=south] {$a_1^3$};
\filldraw[red] (3,9) circle (2pt) node[anchor=west] {$v_1$};

\filldraw[black] (1,8) circle (2pt) node[anchor=east] {};
\filldraw[black] (2,8) circle (2pt) node[anchor=south] {$a_2^3$};
\filldraw[black] (3,8) circle (2pt) node[anchor=south] {$a_2^1$};
\filldraw[black] (4,8) circle (2pt) node[anchor=east] {};
\filldraw[black] (6,8) circle (2pt) node[anchor=south] {$a_4^3$};

\filldraw[black] (5,7) circle (2pt) node[anchor=west] {};
\filldraw[black] (2,7) circle (2pt) node[anchor=east] {};

\filldraw[black] (3.5,6) circle (2pt) node[anchor=west] {};
\draw[->,line width=1pt, blue] (3.5,5.5) to (3.5,5.9);

\filldraw[black] (3.3,4.7)  node[anchor=south] {$w_T =\infty$};
\filldraw[black] (11.3,4.7)  node[anchor=south] {$w_T =6$};

\draw[gray, thick] (10,10) -- (11,9);
\draw[gray, thick] (11,10) --  (11,9);
\draw[gray, thick] (12,10) -- (11,9);

\draw[gray, thick] (8,9) -- (9,8);
\draw[gray, thick] (10,9) --  (9,8);
\draw[gray, thick] (13,9) --  (12,8);
\draw[gray, thick] (11,9) -- (12,8);

\draw[gray, thick] (9,8) -- (10,7);
\draw[gray, thick] (10,8) --  (10,7);
\draw[gray, thick] (11,8) -- (10,7);
\draw[gray, thick] (12,8) --  (13,7);
\draw[gray, thick] (14,8) -- (13,7);

\draw[gray, thick] (13,7) --  (11.5,6);
\draw[gray, thick] (10,7) -- (11.5,6);

\filldraw[black] (10,10) circle (2pt) node[anchor=south] {$a^1_1$};
\filldraw[black] (11,10) circle (2pt) node[anchor=south] {$a^1_4$};
\filldraw[black] (12,10) circle (2pt) node[anchor=south] {$a^1_2$};

\filldraw[black] (8,9) circle (2pt) node[anchor=south] {$a^2_1$};
\filldraw[black] (10,9) circle (2pt) node[anchor=south] {$a^2_2$};
\filldraw[black] (13,9) circle (2pt) node[anchor=south] {$a_3^1$};
\filldraw[red] (11,9) circle (2pt) node[anchor=west] {$v_3$};

\filldraw[red] (9,8) circle (2pt) node[anchor=east] {$v_2$};
\filldraw[black] (10,8) circle (2pt) node[anchor=south] {$a_2^3$};
\filldraw[black] (11,8) circle (2pt) node[anchor=south] {$a^1_5$};
\filldraw[black] (12,8) circle (2pt) node[anchor=east] {};
\filldraw[black] (14,8) circle (2pt) node[anchor=south] {$a^1_6$};

\filldraw[black] (13,7) circle (2pt) node[anchor=west] {};
\filldraw[black] (10,7) circle (2pt) node[anchor=east] {};

\filldraw[black] (11.5,6) circle (2pt) node[anchor=west] {};
\draw[->,line width=1pt, blue] (11.5,5.5) to (11.5,5.9);

\end{tikzpicture}
\caption{The tree on the left-hand side is mixing. Indeed, descendants of the vertex $v_1$ belong to at least two multisets: $A_1$ and $A_3$. Observe that the tree on the right-hand side is not mixing. Indeed, there are two vertices of height equal to one: $v_2 $ and $v_3$. All descendants of $v_2$ belong to $A_2$ and all descendants of $v_3$ belong to~$A_1$.} \label{fig2}
\end{figure}
\end{ex}

\subsection{The proof of Proposition \ref{przedostatnie2}.}
Let us recall the statement of proposition.

\proppp*

The proof of Proposition \ref{przedostatnie2} is divided into two cases: either a forest $F$ is not mixing, \emph{i.e.} $w_F =\infty$ (\cref{lem1}), or a forest $F$ is mixing, \emph{i.e.} $w_F \neq\infty$ (\cref{lem2}). 
The next two subsections establish these two cases.

\subsection{Proof of the not mixing case.}

\begin{lem}
\label{lem1}
Let $A_1, \ldots , A_n $ be multisets consisting of elements of $\mathcal{A}$. Let $A$ be a sum of those multisets. For any reduced forest $F$ witch is not mixing, we have
\begin{equation*}
\sum_{c \in \mathcal{C}_F } {\left( -1 \right) }^{|c|} =0.
\end{equation*}
\end{lem}

\begin{proof}
Since $F$ is not mixing, there exists a vertex $v$ such that all of its descendants are leaves, and all of them belong to just one multiset $A_i$ for some $i\in [n]$. Consider the following partition of set $\mathcal{C}_F$: 
\begin{equation*}
\left\lbrace \mathcal{C}^k_i \right\rbrace_{i \in \mathbb{Z}}^{k=1,2}
\end{equation*}
where each $\mathcal{C}^1_i$ consists of all $c \in \mathcal{C}_F$ with $|c|=i$ and where the vertex $v$ is coloured by its own colour; $\mathcal{C}^2_i$ consists of all $c \in \mathcal{C}_F$ with $|c|=i$ and where there is another vertex coloured by the same colour as the vertex $v$. We express the sum over all $c \in \mathcal{C}_F$ as follows:
\begin{equation*}
\sum_{c\in \mathcal{C}_F} {\left( -1\right) }^{ |c| } = \sum_{i\in \mathbb{Z}}
\left( \sum_{c\in \mathcal{C}_i^1} {\left( -1\right) }^{ i } + \sum_{c\in \mathcal{C}_i^2} {\left( -1\right) }^{ i } \right) = 
 \sum_{i\in \mathbb{Z}} {\left( -1\right) }^{ i }
\left( \sum_{c\in \mathcal{C}_i^1} 1 - \sum_{c\in \mathcal{C}_{i-1}^2} 1\right) .
\end{equation*}
We will show the equipotency of the sets $\mathcal{C}_i^1$ and $\mathcal{C}_{i-1}^2$ from which it follows that the sum above is equal to $0$ and the statement of the lemma is true. 

Let us construct a bijection between $\mathcal{C}_i^1$ and $\mathcal{C}_{i-1}^2$. Take any $c\in \mathcal{C}_i^1$. Suppose that the vertex $v$ is coloured by $k$. Observe that $k \geq 2$. Indeed, if $v$ was coloured by $1$, it would be the only vertex of this colour. Then, the only $1$-coloured vertex would have descendants belonging to just one multiset $A_i$, which is in contradiction with the fact that $c\in\mathcal{C}_F$ (\emph{i.e.} $c$ is a weakly-mixing colouring). From $c\in \mathcal{C}_i^1$ we construct $c' \in \mathcal{C}_{i-1}^1$ as follows:
\begin{enumerate}
\item keep the colours of vertices coloured by $1,\ldots, k-1$ unchanged,
\item change the colours of vertices coloured by $k,\ldots, i$ to $k-1,\ldots, i-1$ respectively.
\end{enumerate}

\noindent
This procedure is reversible. Indeed, take $c' \in \mathcal{C}_{i-1}^1$ and suppose that vertex $v$ is coloured by $k$ for some $k \geq 1$. Then $c \in \mathcal{C}_{i}^2$ can be recovered by the following procedure:
\begin{enumerate}
\item do not change colours of the vertices coloured by $1,\ldots, k-1$,
\item do not change the colour of $v$, 
\item change the colours of the vertices coloured by $k,\ldots, i-1$ to $k+1,\ldots, i$ respectively (excluding vertex $v$).
\end{enumerate}
\end{proof}

\subsection{How to prove the mixing case?}

We will prove the following lemma.

\begin{lem}
\label{lem2}
Let $A_1, \ldots , A_n $ be multisets consisting of elements of $\mathcal{A}$. Let $A$ be a sum of those multisets. For any reduced mixing forest $F$, we have:
\begin{equation*}
\sum_{c \in \mathcal{C}_F } {\left( -1 \right) }^{|c|} =\left( -1\right)^{w_F}.
\end{equation*}
\end{lem}

To prove the lemma above we show a bijection between gap-free colourings of reduced trees $T \in\Ph{A}$ and gap-free colourings of reduced forests $F\in \Pht{A}$, which are not trees (see \cref{rem}). 
Using this bijection we can restrict the proof of \cref{lem2} just to trees.  
For reduced trees and their gap-free colourings we define a projection of this colourings (see \cref{bez}). 
We make use of the notion of projection in \cref{lemma}.  
Proof of \cref{lem2} is done by induction on number of vertices in tree $T$ and presented in \cref{pp}. 

\subsection{Restriction to the trees.}

\begin{rem}
\label{rem}
There is a natural bijection $f$ between all reduced trees $T \in\Ph{A}$ and all reduced forests $F\in \Pht{A}$, which are not trees. This bijection is obtained by deleting the root of $T$ (see \cref{fig0}). 
Moreover, for a given reduced tree $T\in \Ph{A}$, there is an obvious bijection $f_T$ between all gap-free colourings of $T$ and all gap-free colourings of the corresponding reduced forest $f(T)$, obtained by keeping the colours of the non-deleted vertices, so that
\begin{equation*}
\vert f_T (c) \vert = \vert c \vert -1 .
\end{equation*}
\noindent
Additionally $f_T$ preserves the property of being a weakly-mixing colouring.

The above statement allows us to prove \cref{lem2} just for the case of trees $T\in \Ph{A}$ and conclude the statement for all forests $F\in \Pht{A}$.
Indeed, suppose that the statement of \cref{lem2} holds for trees. 
Consider a mixing forest $F\in \Pht{A}$ which is not a tree. Then, 
the tree $T:=f^{-1} (F) \in \Ph{A}$ is also mixing, hence we can use the statement of \cref{lem2}. 
Observe that $w_F = w_T -1$. 
Using bijections $f$ and $f_T$ we get the following equality

\begin{equation*}
\sum_{c \in \mathcal{C}_F } {\left( -1 \right) }^{|c|}  =
\sum_{c \in \mathcal{C}_{f^{-1}(F)} } {\left( -1 \right) }^{|f_T^{-1} (c)| +1} =
-\sum_{c \in \mathcal{C}_T } {\left( -1 \right) }^{|c|} =
- \left( -1\right)^{w_T} =
\left( -1\right)^{w_F},
\end{equation*}
\noindent
which is the statement of \cref{lem2} for the mixing forest $F\in \Pht{A}$.

\end{rem}

\subsection{Projection of a gap-free colouring.}

For any reduced tree $T\in \Ph{A}$ we consider sub-trees $T_1 ,\ldots, T_k$ formed by deleting the root of $T$. Number $k$ is equal to the degree of the root. Every  sub-tree $T_i$ is also a reduced tree. Every gap-free colouring $c$ induces also a sub-colourings $\bar{c}_1,\ldots, \bar{c}_k$ on  $T_1 ,\ldots, T_k$. Observe that sub-colourings obtained this way are not necessarily gap-free. However, there is a canonical way to make them gap-free.

\begin{de}
\label{bez}
Let $T$ be a reduced tree with a gap-free colouring $c$. 
Let $\bar{c}_1,\ldots, \bar{c}_k$ be the induced colourings on sub-trees $T_1 ,\ldots,T_k$ formed by deleting the root of $T$. 
For some $i \in [k]$, let $j^i_0 <\cdots < j^i_l$ be the sequence of colours used in the colouring $\bar{c}_i$. 
By replacing each $j^i_n$ by $n$ in the colouring $\bar{c}_i$ we obtain a gap-free colouring, which we denote by $c_i$.
We say that $c_i$ as an \emph{$i$-th projection of the colouring} $c$ and denote it as $\p_i (c): =c_i$, see \cref {fig3}.
\end{de}

\begin{figure}
\centering
\begin{tikzpicture}[scale=0.8]

\draw[gray, thick] (2,10) -- (3,9);
\draw[gray, thick] (3,10) --  (3,9);
\draw[gray, thick] (4,10) -- (3,9);

\draw[gray, thick] (0,9) -- (1,8);
\draw[gray, thick] (2,9) --  (1,8);
\draw[gray, thick] (5,9) --  (4,8);
\draw[gray, thick] (3,9) -- (4,8);

\draw[gray, thick] (1,8) -- (2,7);
\draw[gray, thick] (2,8) --  (2,7);
\draw[gray, thick] (3,8) -- (2,7);
\draw[gray, thick] (4,8) --  (5,7);
\draw[gray, thick] (6,8) -- (5,7);

\draw[gray, thick] (5,7) --  (3.5,6);
\draw[gray, thick] (2,7) -- (3.5,6);

\filldraw[black] (2,10) circle (2pt) node[anchor=east] {};
\filldraw[black] (3,10) circle (2pt) node[anchor=west] {};
\filldraw[black] (4,10) circle (2pt) node[anchor=west] {};

\filldraw[black] (0,9) circle (2pt) node[anchor=east] {};
\filldraw[black] (2,9) circle (2pt) node[anchor=west] {};
\filldraw[black] (5,9) circle (2pt) node[anchor=west] {};
\filldraw[black] (3,9) circle (2pt) node[anchor=west] {$1$};

\filldraw[black] (1,8) circle (2pt) node[anchor=east] {$2$};
\filldraw[black] (2,8) circle (2pt) node[anchor=west] {};
\filldraw[black] (3,8) circle (2pt) node[anchor=west] {};
\filldraw[black] (4,8) circle (2pt) node[anchor=east] {$3$};
\filldraw[black] (6,8) circle (2pt) node[anchor=west] {};

\filldraw[black] (5,7) circle (2pt) node[anchor=west] {$4$};
\filldraw[black] (2,7) circle (2pt) node[anchor=east] {$4$};

\filldraw[black] (3.5,6) circle (2pt) node[anchor=west] {$5$};
\draw[->,line width=1pt, blue] (3.5,5.5) to (3.5,5.9);

\draw[gray, thick] (10,10) -- (11,9);
\draw[gray, thick] (11,10) --  (11,9);
\draw[gray, thick] (12,10) -- (11,9);

\draw[gray, thick] (8,9) -- (9,8);
\draw[gray, thick] (10,9) --  (9,8);
\draw[gray, thick] (13,9) --  (12,8);
\draw[gray, thick] (11,9) -- (12,8);

\draw[gray, thick] (9,8) -- (10,7);
\draw[gray, thick] (10,8) --  (10,7);
\draw[gray, thick] (11,8) -- (10,7);
\draw[gray, thick] (12,8) --  (13,7);
\draw[gray, thick] (14,8) -- (13,7);

\filldraw[black] (10,10) circle (2pt) node[anchor=east] {};
\filldraw[black] (11,10) circle (2pt) node[anchor=west] {};
\filldraw[black] (12,10) circle (2pt) node[anchor=west] {};

\filldraw[black] (8,9) circle (2pt) node[anchor=east] {};
\filldraw[black] (10,9) circle (2pt) node[anchor=west] {};
\filldraw[black] (13,9) circle (2pt) node[anchor=west] {};
\filldraw[black] (11,9) circle (2pt) node[anchor=west] {$1$};

\filldraw[black] (9,8) circle (2pt) node[anchor=east] {$2$};
\filldraw[black] (10,8) circle (2pt) node[anchor=west] {};
\filldraw[black] (11,8) circle (2pt) node[anchor=west] {};
\filldraw[black] (12,8) circle (2pt) node[anchor=east] {$3$};
\filldraw[black] (14,8) circle (2pt) node[anchor=west] {};

\filldraw[black] (13,7) circle (2pt) node[anchor=west] {$4$};
\draw[->,line width=1pt, blue] (13,6.5) node[anchor=north,black] {$T_2$} to (13,6.9);
\filldraw[black] (10,7) circle (2pt) node[anchor=east] {$4$};
\draw[->,line width=1pt, blue] (10,6.5) node[anchor=north,black] {$T_1$} to (10,6.9);

\draw[gray, thick] (2,3) -- (3,2);
\draw[gray, thick] (3,3) --  (3,2);
\draw[gray, thick] (4,3) -- (3,2);

\draw[gray, thick] (0,2) -- (1,1);
\draw[gray, thick] (2,2) --  (1,1);
\draw[gray, thick] (5,2) --  (4,1);
\draw[gray, thick] (3,2) -- (4,1);

\draw[gray, thick] (1,1) -- (2,0);
\draw[gray, thick] (2,1) --  (2,0);
\draw[gray, thick] (3,1) -- (2,0);
\draw[gray, thick] (4,1) --  (5,0);
\draw[gray, thick] (6,1) -- (5,0);

\filldraw[black] (2,3) circle (2pt) node[anchor=east] {};
\filldraw[black] (3,3) circle (2pt) node[anchor=west] {};
\filldraw[black] (4,3) circle (2pt) node[anchor=west] {};

\filldraw[black] (0,2) circle (2pt) node[anchor=east] {};
\filldraw[black] (2,2) circle (2pt) node[anchor=west] {};
\filldraw[black] (5,2) circle (2pt) node[anchor=west] {};
\filldraw[black] (3,2) circle (2pt) node[anchor=west] {$1$};

\filldraw[black] (1,1) circle (2pt) node[anchor=east] {$1$};
\filldraw[black] (2,1) circle (2pt) node[anchor=west] {};
\filldraw[black] (3,1) circle (2pt) node[anchor=west] {};
\filldraw[black] (4,1) circle (2pt) node[anchor=east] {$2$};
\filldraw[black] (6,1) circle (2pt) node[anchor=west] {};

\filldraw[black] (5,0) circle (2pt) node[anchor=west] {$3$};
\draw[->,line width=1pt, blue] (5,-0.5) node[anchor=north,black] {$T_2$} to (5,-0.1) ;
\filldraw[black] (2,0) circle (2pt) node[anchor=east] {$2$};
\draw[->,line width=1pt, blue] (2,-0.5) node[anchor=north,black] {$T_1$} to (2,-0.1);

\draw[->,line width=1pt] (9,-1) to (9,2) node[anchor=south] {$T_1$};
\draw[->,line width=1pt] (9,-1) to (13,-1) node[anchor=west] {$T_2$};

\filldraw[black] (10,-1) circle (1pt) node[anchor=north] {$1$};
\filldraw[black] (11,-1) circle (1pt) node[anchor=north] {$2$};
\filldraw[black] (12,-1) circle (1pt) node[anchor=north] {$3$};
\filldraw[black] (9,0) circle (1pt) node[anchor=east] {$1$};
\filldraw[black] (9,1) circle (1pt) node[anchor=east] {$2$};

\draw[-, red, line width=1.5pt] (9,-1) -- (10,-1) ;
\filldraw[red] (10,-1) circle (2.5pt);
\draw[-, red, line width=1.5pt] (10,-1) --  (10,0);
\filldraw[red] (10,0) circle (2.5pt);
\draw[-, red, line width=1.5pt] (10,0) -- (11,0);
\filldraw[red] (11,0) circle (2.5pt);
\draw[-, red, line width=1.5pt] (11,0) --  (12,1);
\filldraw[red] (12,1) circle (2.5pt);

\node[text width=5.5cm] at (3,4.8) 
    {(a) Colouring $c$ of a tree $T$ uses the colours $\{0,1,2,3,4,5 \}$.};
\node[text width=6cm] at (11.5,4.5) 
    {(b) Colouring $\bar{c}_1$ of the tree $T_1$ uses the colours $\{0,2,4\}$. Colouring $\bar{c}_2$ of the tree $T_2$ uses the colours $\{0,1,3,4\}$.};    
    
\node[text width=5.5cm] at (3,-2) 
    {(c) Projections $\p_1 (c)$ and $\p_2 (c)$ use colours $\{0,1,2\}$ and $\{0,1,2,3\}$ respectively.};
\node[text width=6cm] at (11.5,-2.2) 
    {(d) Path $\rho$ constructed from $c$ in the proof of Lemma \ref{lemma}.};

\end{tikzpicture}
\caption{(a) A reduced tree $T$ with a gap-free colouring $c$. (b) By deleting the root we obtain two reduced subtrees: $T_1$ and $T_2$ with inherited colourings: $\bar{c}_1$ and $\bar{c}_2$. Observe that they are not gap-free. (c) However, the procedure given in Definition \ref{bez} describes the canonical way of producing a gap-free colourings $\p_1 (c)$ and $\p_2 (c)$. (d) Moreover, for the colouring $c$ we present an associated path $\rho$ which will be introduced in a proof of Lemma \ref{lemma}.} \label{fig3}
\end{figure}
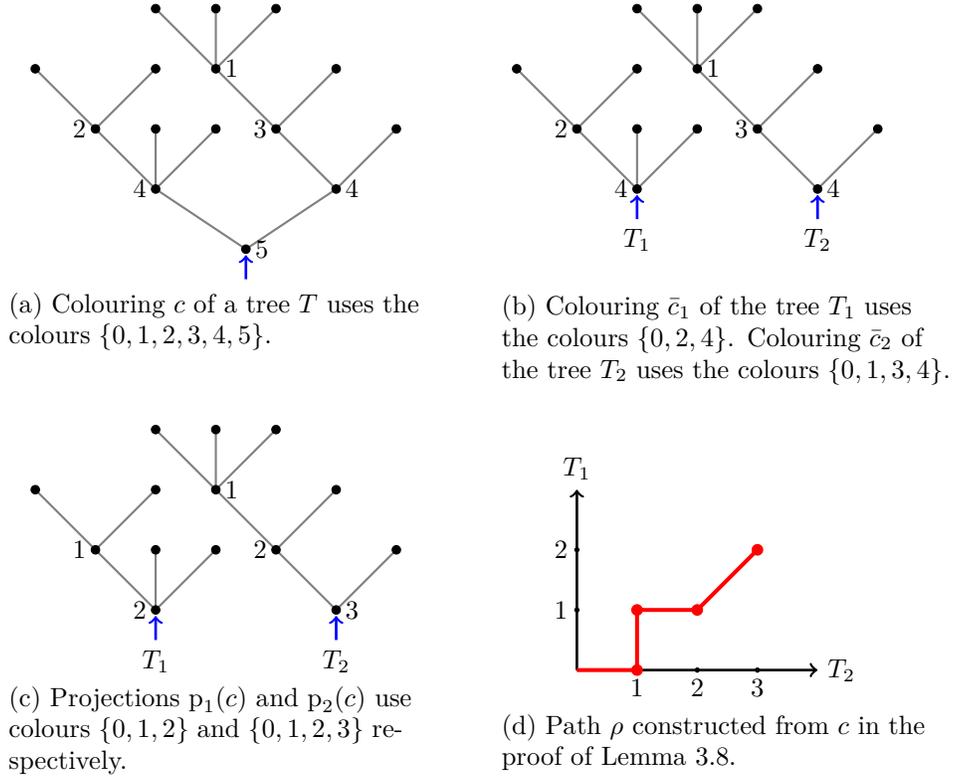

\begin{lem}
\label{lemma}
Let $T$ be a reduced mixing tree of height $h(T) \geq 2$. Denote by $T_1, \ldots,T_r$ all sub-trees obtained by deleting the root of $T$. Let $c_1, \ldots ,c_r$ be gap-free colourings of $T_1, \ldots,T_r$ respectively. Then the following equality holds:
\begin{equation*}
\sum_{\substack{c \in \mathcal{C}_T \\  \p_i(c)= c_i}} {\left( -1 \right) }^{|c|} =
- \prod_{i=1}^r {\left( -1 \right) }^{|c_i|} .
\end{equation*}
\end{lem}

\begin{proof}[Proof of Lemma \ref{lemma}]
\label{rho}
Observe that if $T$ is the mixing tree, then any gap-free colouring $c$ belongs to $\mathcal{C}_T$. Indeed, take any vertex $v$ coloured by $1$. Clearly, its descendants are leaves labelled by elements of at least two distinct multisets $A_i$ and $A_j $ (by assumption that $T$ is mixing). The existence of such a vertex implies that $c\in \mathcal{C}_T$. 

The proof is divided into three steps: first, we construct a bijection between the gap-free colouring $c$ projecting onto $c_1 ,\ldots , c_r$ and some integer paths in $\mathbb{N}^r$; second, we introduce a generating function of these paths and characterise it by recursion on the endpoints and some boundary condition; finally, we find a function satisfying those conditions.\hfill \break

\indent
$\textit{Step 1.}$
Let us recall that any gap-free colouring $c$ of $T$ induces colourings $\overline{c}_i $ on $T_i$, from which we deduce a gap-free colouring $c_i$ of $T_i$ (Definition \ref{bez}). We shall construct a bijection between all gap-free colourings $c$ of $T$ projecting on $c_1,\ldots,c_r$ and all integer paths $\rho $ such that:
\begin{itemize}
\item $\rho$ connects $(0, \ldots, 0) $ and $\left( |c_1|, \ldots,|c_r|\right)  \in \mathbb{N}^r$,
\item each step of $\rho$ is of the following form: 
$\left( k^{n}_1, \ldots,k^{n}_r\right) \in \{0,1\}^r \setminus \left( 0,\ldots ,0 \right)  .$
\end{itemize}
\noindent
Denote the class of such paths by $\mathcal{P}_{ |c_1|, \ldots,|c_r|}$. Moreover the construction is done in such a way that $|c| =|\rho |+1$, where by $|\rho |$ we denote the number of steps in $\rho$. \hfill \break

For a gap-free colouring $c$ we construct a path $\rho$ starting from $(0, \ldots, 0) \in \mathbb{N}^r$ by the following procedure: the $n$-th step of $\rho$ is of the form $\left( k^n_1, \ldots, k^n_r \right) $ where:
\begin{equation*}
k^n_i =
\begin{cases}
1    &$if colour $n$ appears in $\overline{c}_i , \\
0    &$if it does not.$
\end{cases}
\end{equation*}
\noindent
An example of such path is presented on Figure \ref{fig3}. 

The procedure described above is reversible. Indeed, take a path $\rho$ between $\left( 0, \ldots, 0\right) $ and $ \left( |c_1|, \ldots, |c_r|\right)  \in \mathbb{N}^r$. Suppose that the $n$-th step is of the form:
\begin{equation*}
\left( k^{n}_1, \ldots,k^{n}_r\right) \in \{0,1\}^r \setminus \left( 0,\ldots ,0 \right)  .
\end{equation*}
\noindent
We can assign to the path $\rho$ a colouring $c$ by the following procedure. Let $\left( x_1,\ldots ,x_r \right)$ 
be an endpoint of $\rho$ after the $n$-th step. 
We colour each vertex $v\in F_i$ by $n$ if $v$ was coloured by $x_i$ in colouring $c_i$ and $k_i^n \neq 0 $. We colour the root by $|\rho | +1$.\hfill \break

\indent
$\textit{Step 2.}$
The bijection from $\textit{Step 1}$ 
was constructed in such a way that $|c| = |\rho | +1$. Observe that
\begin{equation*}
\sum_{\substack{c \in \mathcal{C}_{T} \\ \p_i (c)=c_i }} \left( -1 \right)^{|c|} =
\sum_{\rho \in \mathcal{P}_{ |c_1|, \ldots, |c_r| }} \left( -1 \right)^{|\rho|+1}.
\end{equation*} 

Let us define a function $F : \mathbb{Z}^r  \longrightarrow \mathbb{Z}$:
\begin{equation*}
F: (x_1, \ldots, x_r)    \longmapsto      
\sum_{\rho \in \mathcal{P}_{x_1,\ldots , x_r}} \left( -1 \right)^{|\rho|+1}.
\end{equation*}

\noindent
Observe that:
\begin{enumerate}[I.]
\item \label{1a} for all $(x_1, \ldots, x_r) \not\in \mathbb{N}^r$, $F(x_1, \ldots, x_r) =0 $, 
\item \label{2a} $F(0, \ldots, 0) =-1$, 
\item \label{3a} for all $(x_1, \ldots, x_r) \in \mathbb{N}^r \setminus (0,\ldots ,0)$, the function $F$ satisfies the following recursive formula:  
\begin{equation*}
F(x_1, \ldots, x_r) = -\mathop{\sum}\limits_{\substack{X \subseteq [r] \\ X\neq \emptyset}}^{} 
 F(\overline{x}^X_1, \ldots, \overline{x}^X_r),
\end{equation*}
\noindent
where $\overline{x}^X_i = \left\lbrace  \begin{smallmatrix} x_i&$ if $ i\not\in X \\ x_i -1&$ if $ i\in X  \end{smallmatrix} \right. $
\end{enumerate}

\indent
Let us shortly comment on this observation. There are no paths connecting $(0, \ldots, 0)$ with points $(x_1, \ldots,x_r) \not\in \mathbb{N}^r$ using the set of steps which is non-negative (Observation \ref{1a}). There is just one path connecting point $(0,\ldots,0)$ to itself: the empty path. Its length is equal to $0$ (Observation \ref{2a}). Consider all possibilities for the last step in path $\rho$. It is equivalent to choosing indices $X \subset [r] $, $X\neq \emptyset$ and summing over all paths ending in $(\overline{x}^X_1,\ldots,\overline{x}^X_r)$ multiplied by $-1$, because we count the sign of the path (Observation \ref{3a}).

Those three statements about function $F$ define it uniquely. 
The recursive formula gives us the way to compute $F(x_1,\ldots,x_r)$ inductively according to $\mathop{\sum}\limits_{i=1}^{r} x_i $. 
The first and the second observation give us the starting point for our induction, namely the values $F(x_1,\ldots,x_r)$ for $\mathop{\sum}\limits_{i=1}^{r} x_i =0$.\hfill \break

\indent
$\textit{Step 3.}$
We show now that the function $G : \mathbb{Z}^r  \longrightarrow \mathbb{Z}$:
\begin{equation*}
G: (x_1,\ldots,x_r) \longmapsto  
\left\lbrace 
\begin{array}{ll}
  - \prod_{i=1}^r \left( -1 \right)^{x_i}    & $if for all $i: x_i\geq0 ,\\
0                                                                           &$otherwise.$
\end{array}  
\right.  
\end{equation*}
satisfies all three properties \ref{1a}, \ref{2a}, \ref{3a} mentioned in $\textit{Step 2.}$ Hence, those two functions $F$ and $G$ are equal. By connecting the results of each step, we get the statement of the lemma:
\begin{equation*}
\begin{array}{rl}
\sum_{\substack{c \in \mathcal{C}_{T} \\ p_i (c)=c_i }} \left( -1 \right)^{|c|} &=
\sum_{\rho \in \mathcal{P}_{( |c_1|, \ldots,|c_r| )}} \left( -1 \right)^{|\rho|+1} =
F(|c_1|, \ldots,|c_r|) = \\
&=G(|c_1|, \ldots,|c_r|) =
- \prod_{i=1}^r \left( -1 \right)^{|c_i|}.
\end{array}
\end{equation*} 

We shall show that function $G$ satisfies three properties \ref{1a}, \ref{2a}, \ref{3a}. Clearly, it satisfies \ref{1a}, \ref{2a}. In order to show that the recursive formula also holds, take any 
$(x_1,\ldots,x_r) $: $\mathop{\forall}\limits_{i}^{} x_i \geq 0 $ and $\mathop{\sum}\limits_{i=1}^{r} x_i >0$.\hfill \break
\indent
Define the set $Y$ consisting of boundary indices $i$ of the point $(x_1,\ldots,x_r) $. More precisely, define $Y \subset [r]$ as follows: $\left\lbrace  \begin{smallmatrix} i \in Y &$ if $ x_i >0 \\ i \not\in Y &$ if $ x_i =0 \end{smallmatrix} \right. $. \hfill \break
\indent
In order to show that $G$ satisfies  \ref{3a}, we have to show the vanishing of the following sum: 
\begin{equation}
\label{van}
\begin{split}
G(x_1,\ldots,x_r) + 
\mathop{\sum}\limits_{\substack{X \subset [r] \\ X\neq \emptyset}}^{} 
G(\overline{x}^X_1,\ldots,\overline{x}^X_r) = 
\mathop{\sum}\limits_{X \subset [r]}^{} 
G(\overline{x}^X_1,\ldots,\overline{x}^X_r) =\\
=\mathop{\sum}\limits_{X \subset Y}^{} 
G(\overline{x}^X_1,\ldots,\overline{x}^X_r) +
\mathop{\sum}\limits_{X \not\subset Y}^{} 
 G(\overline{x}^X_1,\ldots,\overline{x}^X_r) .
\end{split}
\end{equation}

Observe that summants of the sum over $X \not\subset Y$ are equal to $0$. From $X \not\subset Y$ it follow that there exists $i \in [r]$ such that $i\in X$ and $i\not\in Y$. It means that $\overline{x}^X_i =-1$ and by definition $G(\overline{x}^X_1,\ldots,\overline{x}^X_r) =0$. \hfill \break
\indent
Observe that the summants in the sum over $X \subset Y$ are of the form $- \prod_{i=1}^r \left( -1 \right)^{\overline{x}^X_i}$. Indeed, from $X \subset Y$ it follows that $\overline{x}^X_i \geq 0$ for all $i \in [r]$. Thus we have 
\begin{equation*}
\begin{split}
\eqref{van} = 
\mathop{\sum}\limits_{X \subset Y}^{} 
- \prod_{i=1}^r \left( -1 \right)^{\overline{x}^X_i} =
-\mathop{\sum}\limits_{X \subset Y}^{} 
\left( -1 \right)^{\mathop{\sum}\limits_{i=1}^{r} x_i -|X|} \\
=\left( -1 \right)^{\mathop{\sum}\limits_{i=1}^{r} x_i } \cdot 
\mathop{\sum}\limits_{i=0}^{|Y|} {{|Y|}\choose{i}} \cdot \left( -1 \right)^{i} =
0.
\end{split}
\end{equation*}
\end{proof}

\subsection{Proof of \cref{lem2}.}
\label{pp}

\begin{rem}
\label{cor1}
Let $T$ be a reduced mixing tree such that $h(T) \leq 2$. Let $T_1, \ldots , T_r$ be sub-trees obtained from $F$ by deleting the roots. For any colouring $c \in \mathcal{C}_F$ the projection $c_i :=\p_i (c)$ is in $ \mathcal{C}_{T_i}$ for each $i\in [r]$.

Indeed, by definition, $c_i =\p_i (c)$ are gap-free colourings of $F_i$. Take any vertex $v$ coloured by $1$ in $c_i$. Its descendants are leaves. Hence $w_F \neq \infty$, so also $k_{F_i} \neq \infty$. That means that descendants of $v$ belong to at least two distinct multisets $A_i$. The existence of such vertex implies that $c\in  \mathcal{C}_{T_i}$.
\end{rem}

\begin{proof}[Proof of Proposition \ref{lem2}]

We will use induction on the height of tree $T\in\Ph{A}$.

We cannot begin with a tree of height $0$, namely a one-vertex tree $T =\bullet$, because there is no such s tree if $|A| \geq 2$.\hfill \break 

\indent
\textit{Induction base.} We begin from a tree $T$ of height one, namely consisting just of the root and leaves. 
We have exactly one gap-free colouring of length $1$, namely leaves are coloured by $0$ and the root by $1$. 
The claim follows immediately. \hfill \break 

\indent
\textit{Induction step.} Let $n \geq 2$. Suppose now that the statement of Lemma \ref{lem2} is true for any tree $T$ of the height $h(T) \leq n-1$. We will show that it is also true for any tree of height equal to $n \geq 2$. Take such a tree $T$. Denote by $T_1 , \ldots , T_r$ its sub-trees obtained from $T$ by deleting the root. Clearly for every $T_i$, the $h(T_i) \leq n-1$ and we can use the induction hypothesis for them. We have:

\begin{align*}
\sum_{c \in \mathcal{C}_T} {\left( -1\right) }^{|c| } 
&\overset{\text{Remark \ref{cor1}}}{=}
\sum_{\substack{c_1 \in \mathcal{C}_{T_1} \\ \cdots \\ c_r \in \mathcal{C}_{T_r} }} \sum_{\substack{c \in \mathcal{C}_T \\ \p_i (c)=c_i}} {\left( -1\right) }^{|c| } \\
& \overset{\text{Lemma \ref{lemma}}}{=} 
- \sum_{\substack{c_1 \in \mathcal{C}_{T_1} \\ \cdots \\ c_r \in \mathcal{C}_{T_r} }} \prod_{i=1}^r {\left( -1\right) }^{|c_i| } 
 =- \prod_{i=1}^r \sum_{c_i \in \mathcal{C}_{T_i} }  {\left( -1\right) }^{ |c_i| }  \\
& \overset{\text{Induction}}{=} 
- \prod_{i=1}^r  {\left( -1\right) }^{ k_{T_i} }   \\
& \overset{\text{Definition \ref{remark}}}{=} 
{\left( -1\right) }^{ k_{T} } , \\
\end{align*}
\noindent
which proves the statement for any tree of height equal to $n$.
\end{proof}

\let\k\oldk
\bibliographystyle{alpha}
\bibliography{Jack_bib}

\end{document}